\documentclass[12pt]{amsart}

\usepackage[round]{natbib}

\usepackage{lineno,hyperref}
\usepackage{epstopdf}
\usepackage{pifont}
\usepackage{dsfont}
\usepackage{tikz}
\usepackage{latexsym}
\usepackage{graphics}
\usepackage{graphicx}
\usepackage{float}
\usepackage[dvips]{xy}
\xyoption{all}
\usepackage{ifpdf}
\usepackage{multirow}
\usepackage{amssymb, amsthm, amsmath}
\usepackage{hhline}
\usepackage{centernot}
\usepackage{verbatim, comment, color}
\usepackage{enumerate}
\usepackage{setspace}
\usepackage{array}
\usepackage{tabularx}

\newcolumntype{b}{X}
\newcolumntype{s}{>{\hsize=.5\hsize}X}

\usepackage{lipsum}
\usepackage{subcaption}
\usepackage{mathtools}

\usepackage{enumitem}

\usepackage{kbordermatrix}

\newtheorem{theorem}{Theorem} 

\newtheorem{definition}[theorem]{Definition}

\theoremstyle{definition}
\newtheorem{example}{Example}[section]
\newtheorem{remark}{Remark}[section]

\newcommand{\ba}{\begin{align}}
\newcommand{\ea}{\end{align}}  

\newcommand{\be}{\begin{equation}}
\newcommand{\ee}{\end{equation}}

\newcommand{\bea}{\begin{eqnarray}}
\newcommand{\eea}{\end{eqnarray}}
\newcommand{\barr}{\begin{array}}
\newcommand{\earr}{\end{array}}
\newcommand{\bn}{\begin{enumerate}}
\newcommand{\en}{\end{enumerate}}
\newcommand{\bi}{\begin{itemize}}
\newcommand{\ei}{\end{itemize}}
\newcommand{\bbbm}{\begin{pmatrix}}
\newcommand{\eeem}{\end{pmatrix}}

\newcommand{\cC}{{\cal C}}

\newcommand{\cE}{{\cal E}}
\newcommand{\cF}{{\cal F}}
\newcommand{\cG}{{\cal G}}

\newcommand{\cI}{{\cal I}}
\newcommand{\cN}{{\cal N}}
\newcommand{\cP}{{\cal P}}

\newcommand{\cV}{{\cal V}}

\newcommand{\E}{{\mathbb E}}

\newcommand{\N}{{\mathbb N}}

\newcommand{\R}{{\mathbb R}}

\newcommand{\al}{\alpha}

\newcommand{\la}{\lambda}
\newcommand{\La}{\Lambda}

\newcommand{\vp}{\varphi}

\newcommand{\ignore}[1]{}{}
\newcommand{\noin}{\noindent}

\newcommand{\nn}{\nonumber}

\newcommand{\p}{{\partial}}

\newcommand{\q}{\quad}

\newcommand{{\QED}}{{\hfill QED} \bigskip}

\renewcommand{\subset}{\subseteq}

\newcommand{\cal}{\mathcal}

\newcommand{\es}{\emptyset}

\definecolor{darkspringgreen}{rgb}{0.09, 0.45, 0.27} 
\definecolor{darkgray}{rgb}{0.66, 0.66, 0.66}

\numberwithin{equation}{section}
\numberwithin{theorem}{section}
\tikzset{
    partial ellipse/.style args={#1:#2:#3}{
        insert path={+ (#1:#3) arc (#1:#2:#3)}
    }
}

\usepackage[bottom=3.0cm, right=3.0cm, left=3.0cm, top=3.0cm]{geometry}

\begin{document}
\title[Cooperative networks and Hodge-Shapley value]
{Cooperative networks and Hodge-Shapley value
}
\thanks{
\em {I express my sincere gratitude for the feedback provided by three anonymous referees, as well as the insightful contributions from participants in numerous seminars and conferences regarding this paper. Special thanks are extended to Benjamin Golub, David Miller, Ari Stern, Erhan Bayraktar, and Ibrahim Ekren for their invaluable comments and support.}}

\date{\today}

\author{Tongseok Lim}
\address{Tongseok Lim: Mitchell E. Daniels, Jr. School of Business
\newline  Purdue University, West Lafayette, Indiana 47907, USA}
\email{lim336@purdue.edu}

\begin{abstract} 
Lloyd Shapley's cooperative value allocation theory stands as a central concept in game theory, extensively utilized across various domains to distribute resources, evaluate individual contributions, and ensure fairness. The Shapley value formula and his four axioms that characterize it form the foundation of the theory. 

Traditionally, the Shapley value is assigned under the assumption that all players in a cooperative game will ultimately form the grand coalition. In this paper, we reinterpret the Shapley value as an expectation of a certain stochastic path integral, with each path representing a general coalition formation process. As a result, the value allocation is naturally extended to all partial coalition states. In addition, we provide a set of five properties that extend the Shapley axioms and characterize the stochastic path integral. Finally, by integrating Hodge calculus, stochastic processes, and path integration of edge flows on graphs, we expand the cooperative value allocation theory beyond the standard coalition game structure to encompass a broader range of cooperative network configurations.

\end{abstract}

\maketitle




\section{Introduction}

Lloyd Shapley's value allocation theory for cooperative games has been one of the most central concepts in game theory.  Shapley value is widely used in many fields, including economics, finance, and machine learning, to allocate resources, assess individual agent contributions, and determine the fairness of payouts. Among many excellent treatises on Shapley value, we refer to a recent treatise by \citet{algaba2019handbook}, in which various authors discuss modern applications of the Shapley value to various game-theoretic and operations-research problems including genetics, social choice and social network, finance, politics, tax games, telecommunication and energy transmission networks, queueing problems,  group decision making, spanning trees, and even aircraft landing fees problem. Recently, researchers have started to utilize the Shapley value in diverse fields such as machine learning \citep{ghorbani2019data, mitchell2022sampling, schoch2022cs, rozemberczki2022shapley}, medicine \citep{rodriguez2019interpretation, smith2021identifying}, and sustainable energy \citep{pang2021correlation}. This shows that Shapley's cooperative value allocation theory remains a vibrant area of research, applied across various contexts, and continues to inspire researchers.

The Shapley value possesses two notable facets: Shapley's renowned value allocation formula and his four defining axioms. These axioms are significant as they establish a framework of fairness criteria for evaluating the worth of a cooperative game and the individual contributions of players. By adhering to these axioms, the distribution of value within a game is perceived to be fair, transparent, and intuitively sensible. Given that the Shapley value emerges as the unique outcome that satisfies these axioms, it emerges as a compelling solution for value allocation in cooperative games. Furthermore, the Shapley formula holds significance as it provides a mathematical method for computing the value of a cooperative game and the contributions of individual players. This formula determines each player's marginal contribution to the overall game value by considering all potential permutations of player coalitions. With a range of desirable properties inherited from the Shapley axioms, the Shapley formula is widely embraced as the preferred approach for evaluating individual player performance and distributing the value of a cooperative game.

In a coalitional TU game $v$, the Shapley value can be assigned under the assumption that all players will eventually form the grand coalition, and the Shapley axioms and formula are then used to determine a fair allocation. In other words, Shapley value does not address how to properly assess individual player contribution and allocate the value of a cooperative game when the players form any non-grand, but partial coalition state. While the Shapley value applied  for each subgame of $v$ could allocate the value $v(T)$ assigned to the coalition $T \subsetneq N$, the formula implicitly assumes that the coalition only grows towards the target $T$, thereby failing to fully capture the entire game structure. 

The objective of this study is to extend the Shapley value allocation theory to a broader cooperative framework. Specifically, applying this new theory to a coalitional TU game enables the generation of value allocations at each partial coalition state, where the allocation at the grand coalition coincides with the Shapley value.

To comprehensively depict the evolution of partial coalitions, we consider a coalition formation process where players can both join and depart from existing coalitions until the desired coalition is achieved. This approach offers a more realistic portrayal of coalition formation processes and serves as the inspiration for the average path integral value allocation formula, which extends the Shapley formula. 

Furthermore, as we develop the new allocation theory using Hodge calculus, stochastic processes, and path integration on graphs, it becomes apparent that this generalization should not be confined to the coalition game graph framework. Consequently, we introduce the Hodge-Shapley value, a versatile allocation scheme applicable to any cooperative network, providing allocation values at any stage of cooperation. Finally, we demonstrate how the Hodge-Shapley value can be effectively computed on any graph, assuming that the underlying cooperative process adheres to a natural property.

The structure of this paper is outlined as follows. In Section \ref{Shapleyreview}, we provide a review of Shapley's cooperative value allocation theory. Section \ref{newformula} introduces our extension of the Shapley formula via the random coalition process. In Section \ref{newaxiom}, we present five properties characterizing the extended value allocation formula. Section \ref{valuegeneralized} introduces the Hodge-Shapley value by utilizing arbitrary edge flows as players' marginal value. Section \ref{Hodgeallocation1} extends our value allocation framework to any cooperative network. In Section \ref{AvePoisson}, we explain how Hodge calculus can be applied to compute the new value allocation. Section \ref{Conclusion} provides concluding remarks, while Section \ref{proofs} offers additional proofs.

\section{Review of Shapley axioms and the Shapley formula}\label{Shapleyreview}
We commence by revisiting the renowned Shapley value allocation theory (\citet{shapley1953value}), which remains a source of inspiration for researchers across diverse fields.

To begin, let $N = \{1,2,...,n\}$ represent the set of players of the coalition games
\be
\cG_N = \{ v : 2^{N} \to \R \ | \ v(\emptyset)=0 \}. \nn
\ee
A coalition game (also called a characteristic function) $v$ is a function on the subsets of $N$, where each  $S \subset N$ represents a coalition of players in $S$, and $v(S)$ represents the value assigned to the coalition $S$,  with the null coalition $\emptyset$ receiving zero value. Given $v \in \cG_N$, Shapley considered the question of how to split
the grand coalition value $v(N) $, known as the Shapley value. It is determined uniquely by the following result. 

\begin{theorem}[\citet{shapley1953value}]
  \label{thm:shapley}
  There exists a unique allocation
  $ v \in \cG_{N} \mapsto \bigl( \phi _i (v) \bigr) _{ i \in N } $ satisfying the following conditions:
\vspace{1mm}
  
\noin{\rm {\boldmath$\cdot$} efficiency:} $ \sum _{ i \in N } \phi _i (v) = v (N) $.
\vspace{1mm}

\noin{\rm {\boldmath$\cdot$} symmetry:} 
    $ v \bigl( S \cup \{ i \} \bigr) = v \bigl( S \cup \{ j \} \bigr)
    $ for all $ S \subset N \setminus \{ i, j \} $ yields
    $ \phi _i (v) = \phi _j (v) $.
\vspace{1mm}

\noin{\rm {\boldmath$\cdot$} null-player:} 
    $ v \bigl( S \cup \{ i \} \bigr) - v (S) = 0 $ for all
    $ S \subset N \setminus \{ i \} $ yields $ \phi _i (v) = 0 $.
\vspace{1mm}

\noin{\rm {\boldmath$\cdot$} linearity:} 
    $ \phi _i ( \alpha v + \alpha ^\prime v ^\prime ) = \alpha \phi _i
    (v) + \alpha ^\prime \phi _i ( v ^\prime ) $ for all
    $ \alpha, \alpha ^\prime \in \mathbb{R} $ and $v,v' \in \cG_{N}$.
\vspace{1mm}

  Moreover, this allocation is given by the following explicit formula:
  \begin{equation}
    \label{eqn:shapley}
    \phi _i (v) = \sum _{ S \subset N \setminus \{ i \} } \frac{ \lvert S \rvert ! \bigl( n - \lvert S \rvert -1 \bigr) ! }{ n ! } \Bigl( v \bigl(  S \cup \{ i \} \bigr) - v (S) \Bigr) .
  \end{equation}
\end{theorem}
The four terms denote different aspects of the allocation.  [efficiency] indicates that the value obtained by the grand coalition is fully distributed among the players; [symmetry] indicates that equivalent players receive equal amounts; [null-player] indicates that a player who contributes no marginal value to any coalition receives nothing; and [linearity] indicates that the allocation is linear in terms of game values. The four conditions outlined above are known as the Shapley axioms, the vector $\bigl( \phi _i (v) \bigr) _{ i \in N } $ is referred to as the Shapley value, and \eqref{eqn:shapley} is denoted as the Shapley formula.

The Shapley formula \eqref{eqn:shapley} can be rewritten as follows:  Suppose the players form the grand coalition by
joining, one-at-a-time, in the order defined by a permutation $\sigma$
of $N$. That is, player $i$ joins immediately after the coalition
$ S^{\sigma}_i= \bigl\{ j \in N : \sigma (j) < \sigma (i) \bigr\}
$ has formed, contributing marginal value
$ v \bigl( S^{\sigma}_i \cup \{ i \} \bigr) - v (S^{\sigma}_i
) $. Then $ \phi _i (v) $ is the average marginal value contributed by
player $i$ over all $ n ! $ permutations $\sigma$, that is,
\begin{align}
  \label{eqn:shapleyPermutation}
  \phi _i (v) = \frac{ 1 }{ n ! } \sum _\sigma \Bigl( v \bigl( S^{\sigma}_i \cup \{ i \} \bigr) - v (S^{\sigma}_i ) \Bigr).  
\end{align}
The well-known glove game below explains the formula \eqref{eqn:shapleyPermutation} in a simple context.
\begin{example}[Glove game]\label{ex:introGlove}
Let $n = 3 $. Suppose player
  $1$ has a left-hand glove, while players $2$ and $3$ each have a
  right-hand glove. A pair of gloves has value $1$, while unpaired gloves have no
  value, i.e., $ v (S) = 1 $ if $S \subset N $ contains player $1$ and at least one of players $2$
  or $3$, and $ v (S) = 0 $ otherwise. The Shapley values are:
  \begin{equation*}
    \phi _1 (v) = \tfrac{ 2 }{ 3 } , \qquad \phi _2 (v) = \phi _3 (v) = \tfrac{ 1 }{ 6 } .
  \end{equation*}
  This is easily seen from \eqref{eqn:shapleyPermutation}: player $1$ contributes
  marginal value $0$ when joining the coalition first (2 of 6 permutations) and marginal value $1$ otherwise (4 of 6 permutations), so $ \phi _1 (v) = \tfrac{ 2 }{ 3 } $. Efficiency and symmetry yield $\phi _2 (v) = \phi _3 (v) = \tfrac{ 1 }{ 6 }$.
 \end{example}
We note that the Shapley value can be readily applied to each coalition $T \subset N$ by employing the Shapley formula to the subgame $v \big|_T$\footnote{$v \big|_T : 2^T \to \R$ denotes the restriction of $v$ to the subsets of $T$, i.e., $v \big|_T(S) = v(S)$ for all $S \subset T$.}. This allocations scheme is termed the extended Shapley value. In the subsequent sections, we explore an alternative cooperative value allocation scheme that seamlessly extends to all coalitions $2^N$ and differs from the Shapley value for partial coalitions $T \subsetneq N$.

\section{Extension of the Shapley value via random coalition processes}\label{newformula}
Before delving into our generalization of the Shapley value to various setups, let us initially conceptualize the Shapley value as the expected outcome of a random coalition formation process. Consider the hypercube graph, or  coalition game graph $G=(\cV, \cE)$, where $\cV$ denotes the set of nodes and $\cE$ the set of edges. This graph is defined by
\begin{equation}\label{oldG}
  \cV := 2 ^{N}, \ \  \cE := \bigl\{ \bigl(  S, S \cup \{ i \} \bigr) \in \cV \times \cV \ | \  S \subset N \setminus \{ i \} ,\ i \in N \bigr\}.
  \end{equation}
Notice that each coalition $S \subset N$ can correspond to a vertex of the unit hypercube in $\R^{N}$. We assume that each edge is oriented in the direction of the inclusion $ S \hookrightarrow S \cup \{ i \}$. We also define the set of reverse (or negatively-oriented) edges 
\be
\cE_- := \bigl\{ \bigl(S \cup \{ i \}, S \bigr) \in \cV \times \cV \ | \  S \subset N \setminus \{ i \} ,\ i \in N \bigr\}.
\ee
The edges in $\cE$ are termed forward/positively-oriented edges. We set $\overline \cE = \cE \cup \cE_-$.
 \begin{figure}
\centering
\begin{subfigure}{.5\textwidth}
  \centering
  \includegraphics[width=.73\linewidth]{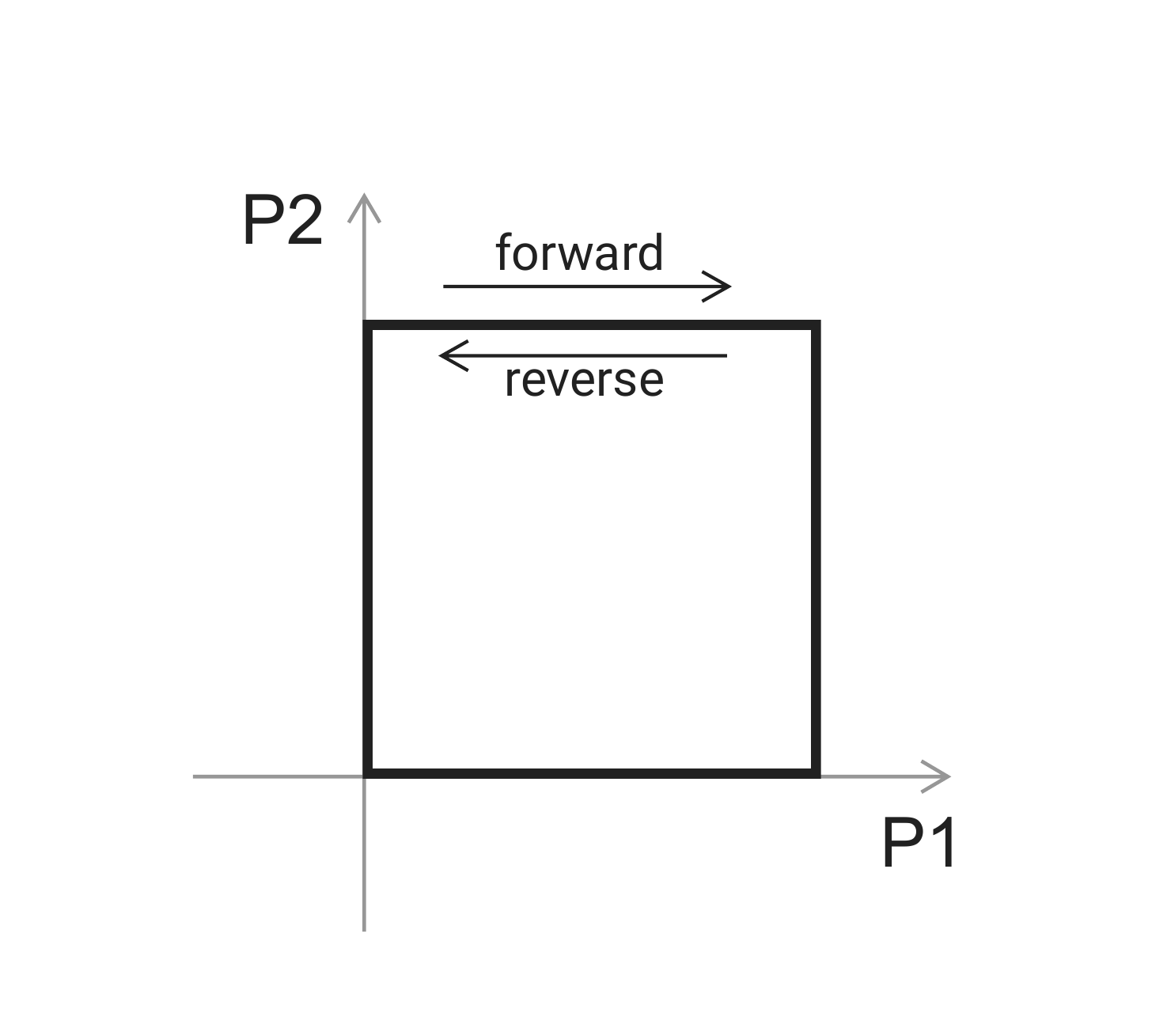}
\end{subfigure}%
\begin{subfigure}{.5\textwidth}
  \centering
  \includegraphics[width=.73\linewidth]{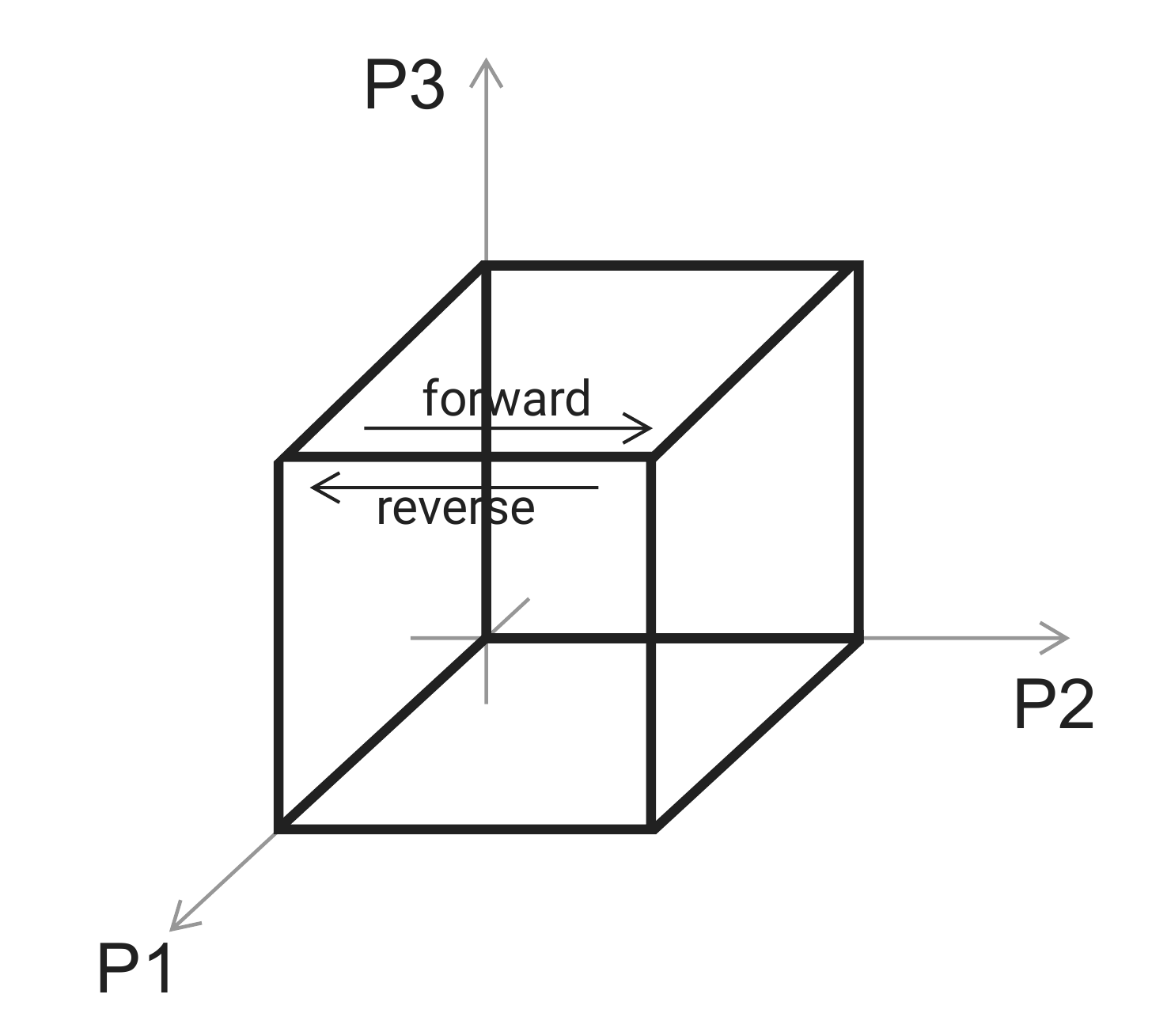}
\end{subfigure}
\caption{Coalition game graph for $n=2$ and $3$. Each vertex of the cube corresponds to a coalition. The vertex $(1,0,1)$, for example, corresponds to the coalition $\{1,3\}$, $(0,1,1)$ to $\{2,3\}$, and so on.}
\label{figure1}
\end{figure}

Next, recall that in the Shapley formula \eqref{eqn:shapleyPermutation}, the coalition formation is supposed to be only increasing, with each step resulting in a player joining a given coalition. In contrast, consider an example of a random coalition process, described by the canonical Markov chain $(X_t)_{t \in \N_0}$ ($\N_0 := \N \cup \{0\}$) on the coalition space $\cV$ in \eqref{oldG} with $X_0 = \emptyset$, equipped with the transition probability $p_{S,T}$ from a state $S$ to $T$ as follows:
 \begin{align}\label{MC}
p_{S,T} := 1/n \ \text{ if } \ T \sim S, \q  p_{S,T} := 0 \ \text{ if } \ T \not\sim S.
\end{align} 
\eqref{MC} may be interpreted as the value allocator's best estimate that each player has a similar likelihood of joining or leaving the given coalition at any time (before any actual coalition begins). We emphasize the coalition process now allows a player to not only join but also leave the current coalition, i.e., the transition from $S$ to $S \setminus \{i\}$ is possible.

Let $(\Omega, \cF, \cP)$ denote the underlying probability space for formality. For each $T \in \cV$ and a sample coalition path $\omega \in \Omega$, let $\tau_{T} = \tau_{T}(\omega) \in \N$ denote the first (random) time the coalition process $\big(X_t(\omega)\big)_t$ visits $T$, that is, $X_{\tau_{T}}(\omega) = T$ and $X_t(\omega) \ne T$ for all $t \le \tau_{T} -1$. 
\begin{figure}
\centering
\begin{subfigure}{.5\textwidth}
  \centering
  \includegraphics[width=.73\linewidth]{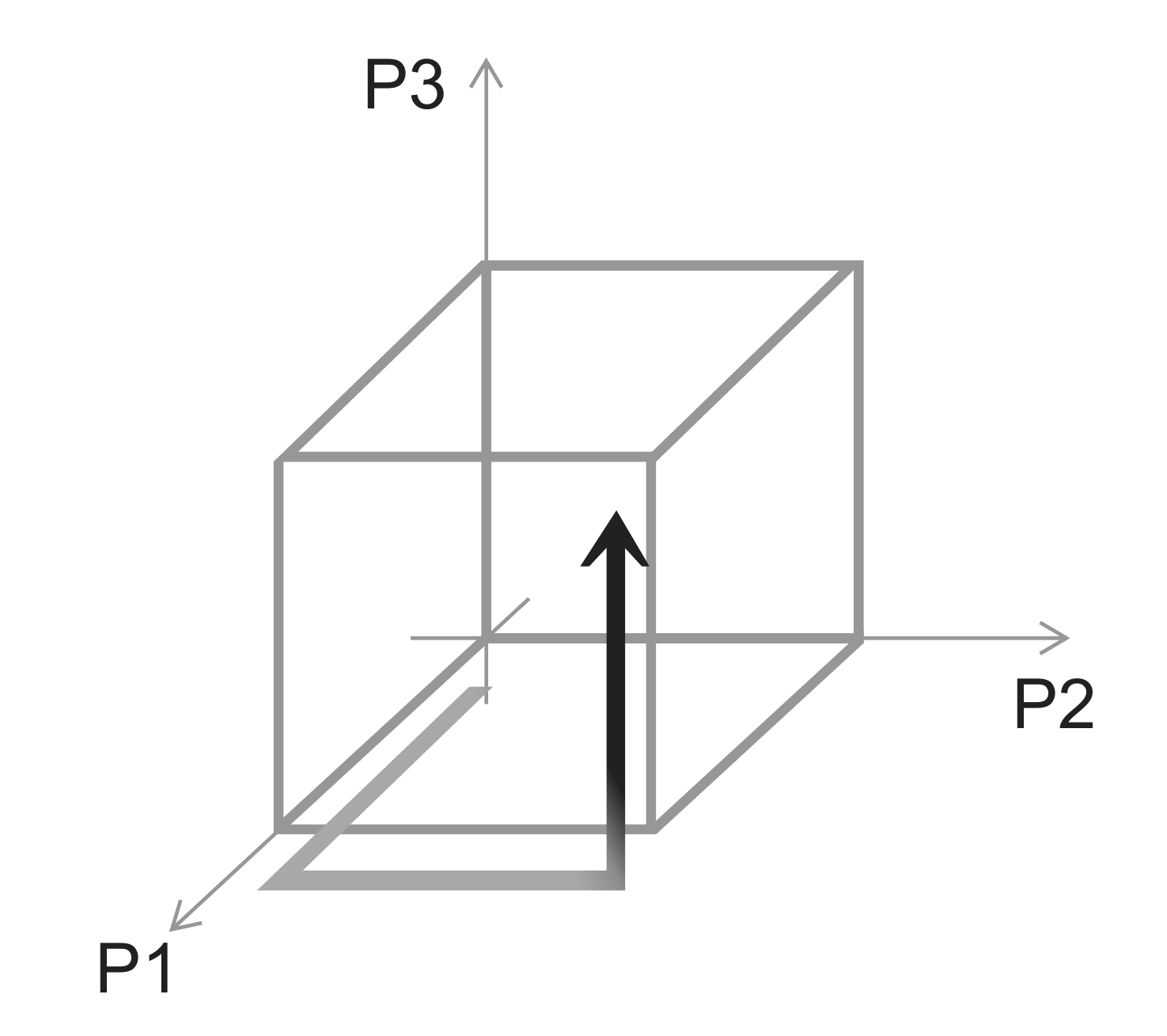}
\end{subfigure}%
\begin{subfigure}{.5\textwidth}
  \centering
  \includegraphics[width=.73\linewidth]{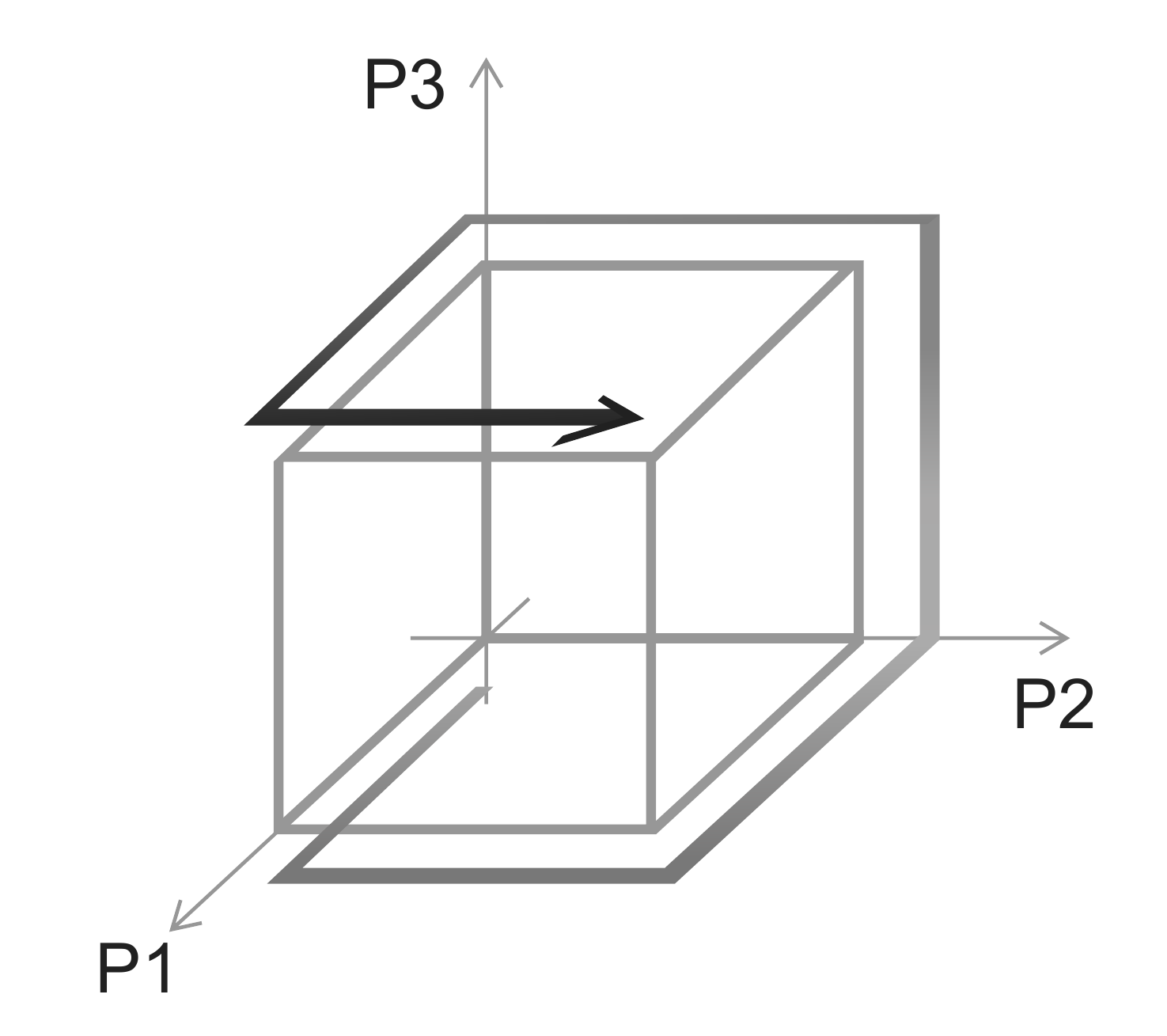}
\end{subfigure}
\caption{Examples of Shapley's coalition path and our general coalition path. The path in this example has no loops, but in general, our coalition path is allowed to have an arbitrary number of loops.}
\label{figure2}
\end{figure}
Now given a coalition game $v \in \cG_N$, the total contribution of player $i$ along the sample coalition path $\omega$ traveling from $\emptyset$ to $T$ can be calculated as
\be\label{pathintegral}
{\cal I}_i(v, T) = {\cal I}_i(v, T)(\omega) := \sum_{t=1}^{\tau_{T}(\omega)} \p_i  v \big(X_{t-1}(\omega), X_t(\omega) \big),
\ee
where for each $ i \in N $, the partial differential $\p_i v : \overline \cE \to \R$ of a game $v$ is defined as
  \begin{equation}\label{partialdifferential}
    \p_i v \bigl( S , S \cup \{ j \} \bigr) :=
    \begin{cases}
v (S \cup \{ i \}) - v(S) & \text{if } \ j=i, \\
      0 & \text{if } \  j \neq i,
    \end{cases}
  \end{equation} 
and $\p_i v \bigl(S \cup \{ j \}, S \bigr) := - \p_i v \bigl( S , S \cup \{ j \} \bigr)$. Notice that $ \p_i v$ represents the marginal contribution of player $i$ to the game $v$ in both directions of joining and leaving. Thus, given that the coalition has progressed to $T$ along the path $\omega$, \eqref{pathintegral} represents player $i$'s total marginal contribution throughout the progression. This in turn implies that the value function given by the following average path integral
\be\label{value}
\Phi_i (v, T) := \E[ {\cal I}_i(v, T)] = \int_\Omega  {\cal I}_i(v, T)(\omega)\, d \cP(\omega)
\ee
represents player $i$'s expected total contribution given the state advances from $\emptyset$ to $T$. 

Now our interpretation of the Shapley value as the expected outcome of the random coalition formation process is stated as follows.
\begin{theorem}\label{coincide1}
The Shapley value and the value function $(\Phi_i)_{i \in N}$ evaluated at the grand coalition coincide. That is, $\phi_i(v) = \Phi_i(v,N)$ for every $v \in \cG_N$ and $i \in N$.
\end{theorem}

Theorem \ref{coincide1} shows the Shapley value $\phi_i(v) $ can be interpreted as the player $i$'s expected total contribution, thus her fair share, if the coalition process ends at the grand coalition and the player $i$'s marginal contribution for each transition is given by $ \p_i v$. 

The summation formulas in \eqref{eqn:shapleyPermutation} and \eqref{value}, on the other hand, appear quite different. While \eqref{eqn:shapleyPermutation} consists of a finite sum along $n!$ paths in increasing order driven by permutations $\sigma$, the sum in \eqref{value} is infinite and takes into account all possible paths $\omega$.
\begin{example}\label{twoperson}
 We demonstrate how to directly calculate the value \eqref{value} for a general two-person game. We denote $v_1 = v(\{1\})$, $v_2 = v(\{2\})$, and $v_{12} = v(\{1,2\})$ for convenience. First, the Shapley formula \eqref{eqn:shapleyPermutation} clearly yields
\be
\phi_1(v) = \tfrac{1}{2}(v_1 + v_{12} - v_2), \q \phi_2(v) = \tfrac{1}{2}(v_2 + v_{12} - v_1).
\nn
\ee
To calculate $\Phi_1(N)$ and $\Phi_2(N)$, note that sample paths from $\emptyset$ to $N=\{1,2\}$ can have lengths (i.e. number of transitions) $2l$, $l \in \N$.  We may enumerate sample paths by their lengths, and calculate the contribution function $\cI_1, \cI_2$ as follows:
\be
    \centering
    \begin{tabularx}{0.9\textwidth}{bss}
        \hline
        Path $\omega$     & $\cI_1(\omega)$     & $\cI_2(\omega)$ \\ \hline
        $(\emptyset, \{1\}, \{1,2\})$        & $v_1$        & $v_{12} - v_1$         \\ \hline
        $(\emptyset, \{2\}, \{1,2\})$         & $v_{12} - v_2$        & $v_2$        \\ \hline
        $(\emptyset, \{1\}, \emptyset, \{1\}, \{1,2\})$     & $v_1$        & $v_{12} - v_1$         \\ \hline
        $(\emptyset, \{1\}, \emptyset, \{2\}, \{1,2\})$     & $v_{12} - v_2$        & $v_2$        \\ \hline
       $(\emptyset, \{2\}, \emptyset, \{1\}, \{1,2\})$     & $v_1$        & $v_{12} - v_1$         \\ \hline
        $(\emptyset, \{2\}, \emptyset, \{2\}, \{1,2\})$     & $v_{12} - v_2$        & $v_2$        \\ \hline
      $(\emptyset, \{1\}, \emptyset, \{1\},\emptyset, \{1\}, \{1,2\})$     & $v_1$        & $v_{12} - v_1$         \\ \hline
    $(\emptyset, \{1\}, \emptyset, \{1\},\emptyset, \{2\}, \{1,2\})$     & $ v_{12}-v_2$        & $v_{2}$         \\ \hline
        $\dots$     & $\dots$   & $\dots$    
    \end{tabularx}
    \nn
\ee
Calculation of $\cI_i$ appears fairly simple due to the sign-changing property of $\p_i$, i.e., $\p_i v \bigl(S \cup \{ j \}, S \bigr) = - \p_i v \bigl( S , S \cup \{ j \} \bigr)$, which yields many cancellations in the path integral $\cI_i(\omega)$. Under the transition law \eqref{MC}, the probability of a sample path $\omega$ of length $L$ being realized is simply $(1/2)^{L}$. This yields
\begin{align*}
\Phi_1(N) &= (\tfrac{1}{2})^2(v_1 + v_{12} - v_2) + (\tfrac{1}{2})^4\cdot 2 (v_1 + v_{12} - v_2) + (\tfrac{1}{2})^6\cdot 2^{2} (v_1 + v_{12} - v_2) + \dots\\
&=\big((\tfrac{1}{2})^2 + (\tfrac{1}{2})^3 + (\tfrac{1}{2})^4 + \dots\big)  (v_1 + v_{12} - v_2) = \tfrac{1}{2}(v_1 + v_{12} - v_2) = \phi_1(v),
\end{align*}
and similarly $\Phi_2(N) = \tfrac{1}{2}(v_2 + v_{12} - v_1) = \phi_2(v)$, as asserted in Theorem \ref{coincide1}.

If the terminal coalition is $\{1\}$, the calculation of  $\Phi_1(\{1\})$ and $\Phi_2(\{1\})$ proceeds as
\be
    \centering
    \begin{tabularx}{0.9\textwidth}{bss}
        \hline
        Path $\omega$     & $\cI_1(\omega)$     & $\cI_2(\omega)$ \\ \hline
        $(\emptyset, \{1\})$        & $v_1$        & $0$         \\ \hline
        $(\emptyset, \{2\}, \es, \{1\})$         & $v_1$        & $0$  \\  \hline
        $(\emptyset, \{2\}, \{1,2\}, \{1\})$     & $v_{12}-v_2$        & $ v_2 - (v_{12} - v_1)$         \\ \hline
        $(\emptyset, \{2\}, \es, \{2\}, \es,  \{1\})$  & $v_1$        & $0$         \\ \hline
        $(\emptyset, \{2\}, \es, \{2\}, \{1,2\},  \{1\})$ & $v_{12}-v_2$ & $v_2 - (v_{12} - v_1)$ \\ \hline
        $(\emptyset, \{2\}, \{1,2\}, \{2\}, \emptyset, \{1\})$ & $v_1$        & $0$         \\ \hline
        $(\emptyset, \{2\}, \{1,2\}, \{2\}, \{1,2\}, \{1\})$ & $v_{12}-v_2$ & $v_2 - (v_{12} - v_1)$ \\ \hline
        $(\emptyset, \{2\}, \es, \{2\}, \es,  \{2\}, \es, \{1\})$ & $v_1$        & $0$  \\  \hline
        $\dots$     & $\dots$   & $\dots$  
    \end{tabularx}
    \nn
\ee
Adding up, we get, for example,
\begin{align*}
\Phi_1(\{1\}) &= \tfrac{1}{2} v_1 + (\tfrac{1}{2})^3 (v_1 + v_{12} - v_2) + (\tfrac{1}{2})^5 \cdot 2 (v_1 + v_{12} - v_2) + \dots\\
&= \tfrac{1}{2} v_1 + \big((\tfrac{1}{2})^3 + (\tfrac{1}{2})^4 + \dots\big)  (v_1 + v_{12} - v_2) = \tfrac{1}{4}(3v_1 + v_{12} - v_2).
\end{align*}
The complete value allocation table is the following.
\be\label{twopersontable}
\centering
\begin{tabularx}{0.9\textwidth}
{ 
  | >{\centering\arraybackslash}X
  | >{\centering\arraybackslash}X 
  | >{\centering\arraybackslash}X 
  | >{\centering\arraybackslash}X 
  | }
\hline 
 & $\{1\}$ & $\{2\}$ & $\{1,2\}$ \\
\hline 
$\Phi_1$ & $\frac14(3v_1 -v_2 + v_{12})$  & $\frac14(v_1+v_2-v_{12})$ & $\frac12(v_1-v_2+v_{12})$  \\
\hline 
$\Phi_2$ & $\frac14(v_1+v_2-v_{12})$  & $\frac14(3v_2 -v_1 + v_{12})$ & $\frac12(v_2-v_1+v_{12})$  \\
\hline
\end{tabularx}
\ee
We observe that efficiency holds, i.e., $v(S) = \Phi_1(S) + \Phi_2(S)$ for all $S \subset N$.
\end{example}

While calculating the value allocation operator $\Phi = (\Phi_i)_{i \in N}$ is feasible for two-person games due to their simple game graph structure, it becomes considerably more challenging as the number of players $n$ increases. For instance, computing $\Phi$ for the glove game quickly becomes nontrivial. In light of this complexity, one might question the rationale for using \eqref{value} when the much simpler Shapley formula is readily available.

The first reason is that $\Phi$ extends value allocation for all partial coalition $T \subset N$. Though the extended Shapley value could allocate the value $v(T)$ by applying the Shapley formula to the subgame $v \big|_T$, the formula implicitly assumes that the coalition only increases toward the target $T$, thus failing to comprehensively reflect the whole game structure. For instance, the Shapley value at $T$ completely disregards marginal value information in the form of $v(S \cup \{i\}) - v(S)$ for any $S \subset T$ and $i \notin T$. In contrast, the value $\Phi$ thoroughly explores the game structure by integrating players' marginal value along each general coalition path, resulting in a fair allocation of the collaborative value $v(T)$ yet distinct from the Shapley value. Nevertheless, Theorem \ref{coincide1} shows that the value \eqref{value} can be interpreted as an extension of the Shapley value.

Secondly, the significance of the Shapley value in cooperative game theory partly lies in the intuitive and compelling nature of its four defining axioms. These axioms have spurred numerous subsequent studies and explorations of variant axioms. Inspired by the literature, we demonstrate that the value $\Phi(v,T)$ can be simultaneously  determined for all $v \in \cG_N$ and $T \subset N$ based on five properties that extend the Shapley axioms.

Thirdly, the average path integral formula \eqref{value} facilitates a significant expansion of cooperative games and their value allocation schemes. This expansion encompasses broader domains such as general cooperative game networks beyond \eqref{oldG}, more comprehensive marginal values of players beyond \eqref{partialdifferential}, and diverse cooperative processes beyond \eqref{MC}. Once we establish a game graph, players' marginal value, and cooperative process on the graph, the formula \eqref{value} becomes meaningful. This generalization is particularly advantageous for investigating allocations of cooperative games that may not inherently satisfy efficiency. The efficiency condition is now interpreted as a straightforward equality constraint on the marginal values of the players; refer to Remark \ref{marginalefficiency}.

Finally, if the formula \eqref{value} proves challenging to compute, it becomes impractical. However, we shall show that \eqref{value} can be easily calculated through a system of linear equations for a broad range of cooperative processes. For instance, the solution to the two-person game in Example \ref{twoperson} can be obtained by solving the system
\begin{align*}
\begin{bmatrix}
-1 & -1 & 0 \\
2 & 0 & -1 \\
0 & 2 & -1 \\
-1 & -1 & 2
\end{bmatrix} 
\begin{bmatrix}
\Phi_1(\{1\}) & \Phi_2(\{1\})\\
\Phi_1(\{2\}) & \Phi_2(\{2\}) \\
\Phi_1(\{1,2\}) & \Phi_2(\{1,2\})
\end{bmatrix} 
&=
\begin{bmatrix}
-v_1 & -v_2\\
v_1 & v_1 - v_{12} \\
 v_2 - v_{12} & v_2 \\
v_{12} - v_2 & v_{12} - v_1
\end{bmatrix}.
\end{align*}
Calculation of \eqref{value} in this manner for the glove game yields the following table.
\be\label{glovegameextended}
\noin
\begin{tabularx}{0.913\textwidth}
{ 
  | >{\centering\arraybackslash}X 
    | >{\centering\arraybackslash}X 
  | >{\centering\arraybackslash}X 
  | >{\centering\arraybackslash}X 
    | >{\centering\arraybackslash}X 
    | >{\centering\arraybackslash}X 
  | >{\centering\arraybackslash}X 
  | >{\centering\arraybackslash}X 
  | }

\hline 
 & $\{1\}$ & $\{2\}$ & $\{3\}$ & $\{1,2\}$ & $\{1,3\}$ & $\{2,3\}$ & $\{1,2,3\}$ \\
\hline 
$\Phi_1$ & $\frac{5 }{ 12}$ & $ -\frac{ 5}{ 24}$ & $ -\frac{ 5}{ 24}$ & $\frac{5 }{ 8} $ &$ \frac{5 }{ 8}$ &$ -\frac{ 1}{ 4} $ & $\frac{2 }{ 3} $ \\
\hline 
$\Phi_2$ &$-\frac{ 5}{ 24} $ & $\frac{1 }{6 } $ & $  \frac{1 }{ 24} $ & $  \frac{ 3}{ 8}  $ & $0 $ & $\frac{1 }{8} $ & $\frac{1 }{ 6} $ \\
\hline
$\Phi_3$ & $-\frac{ 5}{ 24} $ & $ \frac{1 }{ 24}  $ & $\frac{1 }{6 }$ & $ 0$ & $ \frac{ 3}{ 8}$ & $\frac{1 }{8}$ & $\frac{1 }{ 6} $ \\
\hline
\end{tabularx}
\ee
The final column indeed corresponds to the Shapley value. Later, we will elaborate on and expand this table for the $\alpha$-Shapley value. Finally, we shall elucidate the connection between the average path integral \eqref{value} and the linear system, referred to as Poisson's equation, by illuminating it through the perspective of combinatorial Hodge theory.

\section{A characterization of the value $\Phi$}\label{newaxiom}

In addition to Shapley's original axiomatic framework, several alternative foundations have been proposed since. \citet{young1985monotonic} showed that the Shapley value is the unique solution that adheres to efficiency, symmetry, and marginality.\footnote{Marginality states that for all $i \in N$, if $\Delta_i v = \Delta_i w$, then $\phi_i(v) = \phi_i(w)$, where $\Delta_i v (S) := v(S) - v(S \setminus \{i\})$ if $i \in S$, $v(S \cup \{i\}) - v(S)$ if $i \notin S$.} \citet{chun1989new} showed that the Shapley value is the only value that satisfies efficiency, triviality,\footnote{Triviality states that if $v \equiv 0$ ($v(S) = 0$ for all $S \subset N$), then the value $\phi_i(v)=0$ for all $i \in N$.} coalitional strategic equivalence,\footnote{Coalitional strategic equivalence states that for all $\es \neq S \subset N$ and $\al \in \R$, if $v=w + w_\al^S$, then $\phi_i(v) = \phi_i(w)$ for all $i \in N \setminus S$, where $w_\al^S(T) := \al$ if $S \subset T$, and $0$ otherwise.} and fair ranking\footnote{Fair ranking states that for any given $T \subset N$, if $v(S)=w(S)$ for all $S \neq T$, then $\phi_i(v) > \phi_j(v)$ implies $\phi_i(w) > \phi_j(w)$ for all $i, j \in T$.} criteria. \citet{casajus2018decomposition} characterizes the Shapley value as the unique decomposable decomposer of the naïve value $\big(v(N) - v(N \setminus \{i\}\big)_{i \in N}$ into a direct part and an indirect part, where the latter indicates how much each player contributes to the other players' direct parts. \citet{hart1989potential}'s characterization of the Shapley value via the potential function approach is discussed in Remark \ref{moreaxioms}. These results motivate us to identify properties that can characterize the value \eqref{value} for all games $v$ and coalitions $S \subset N$ at the same time. Following \cite{casajus2018decomposition}, we assume that the player sets are subsets of a countably infinite set $\cal U$, the universe of players; $\cN$ denotes the set of all finite subsets of $\cal U$. Let $ \cG = \bigcup_{N \in \cN} \cG_N$ denote the set of all coalition games. 

For $i,j \in N$ and $S \subset N$, we define $S^{ij} \subset N$ by switching $i$ and $j$ in $S$, that is,
  \begin{equation*}
S^{ij}=
    \begin{cases}
      S & \text{if } \ S \subset N \setminus \{i,j\}  \, \text{ or } \,  \{i,j\} \subset S, \\
      S \cup \{i\} \setminus \{j\} & \text{if } \ i \notin S \,\text{ and } \, j \in S, \\
S \cup \{j\} \setminus \{i\} & \text{if } \ j \notin S \, \text{ and } \, i \in S.
    \end{cases}
  \end{equation*} 
For $v \in \cG_N$ and $i,j \in N$, we define $v^{ij} \in \cG_N$ by $v^{ij}(S) = v(S^{ij})$. Intuitively, the contributions of the players $i,j$ in the game $v$ are interchanged in the game $v^{ij}$. Let $v_{-i} : 2^{N \setminus \{i\}} \to \R$ denote the restricted game of $v$ on the subsets of players $N \setminus \{i\}$, i.e., $v_{-i}(S) = v(S)$ for all $S \subset N \setminus\{i\}$. We shall now describe our set of five properties.
 \vspace{1mm}

{\rm {\bf A1}(efficiency):} $v (S) = \sum_{i \in N} \Phi _i (v,S)$ for any $v \in \cG_N$ and $S \subset N$.
\vspace{1mm}

{\rm {\bf A2}(linearity):} For any $v, v' \in\cG_{N}$, $\al, \al' \in \R$ and $S \subset N$, it holds
\[
\Phi _i (\al v + \al' v', S) = \al \Phi _i (v,S) + \al' \Phi _i (v',S).
\]
We observe that A1, A2 are natural extensions of Shapley's efficiency and linearity axioms now holding for all coalitions $S \subset N$.
 \vspace{1mm}
 
{\rm {\bf A3}(symmetry):} $ \Phi _i (v^{ij}, S^{ij}) =\Phi _j (v, S) $ for all $v \in \cG_N$, $i,j \in N$ and $S \subset N$.
\vspace{1mm}
 
We may interpret A3 as follows: if the contributions of players $i$ and $j$ are interchanged in the game, their payoffs also switch accordingly.
 \vspace{1mm}

{\rm {\bf A4}(null-player):} For any $v \in \cG_N$ and $i \in N$, if $\p_i v \equiv 0$, then  
\be
\Phi_j(v, S \cup \{i\}) = \Phi_j(v,S) =\Phi_j(v_{-i},S) \ \text{ for all }  j \in N \setminus \{i\} \text{ and }  S \subset N \setminus \{i\}. \nn
\ee
 A4 says if player $i$ provides no marginal value, the reward of the rest is independent of the player $i$'s participation. A4 and A1 implies that in this case, $i$ receives nothing, i.e., $\Phi_i (v, S)= 0$ for all $S \subset N$. Notice A4 gives a relation between the payoffs of $v$ and $v_{-i}$. 
 
So far, A1--A4 can be seen as a natural extension of the Shapley axioms to deal with different groups of players $N$ and coalitions $S$, as well as their symmetric counterpart $S^{ij}$. In particular, A1--A4 will be able to  determine the Shapley value $\Phi_i (v, N)$. However, A1--A4 appears insufficient to fully determine $\Phi$ for all coalitions $S$. Our observation is that the following condition appears to be the key to complement A1--A4.  
\vspace{1mm}

{\rm {\bf A5}(reflection):}
 For any $v \in \cG_N$, $i \in N$ and $S,T \subset N \setminus \{i\}$, it holds  
\be\label{reflect}
\Phi _i (v,  T \cup \{ i \}) - \Phi _i (v, S \cup \{ i \})= -\bigl( \Phi _i (v, T) - \Phi _i (v, S) \bigr). 
\ee
A5 is indeed inspired by the stochastic path integral representation of the value function \eqref{value}. Let $S,T \subset N \setminus \{i\}$, and consider an arbitrary coalition path (see Figure \ref{figure3})
\be
\omega: X_0 \to X_1 \to \dots \to X_\tau \nn
\ee
where $X_0 = S$, $X_\tau = T$, and each $(X_t, X_{t+1})$ is either a forward- or reverse-oriented edge of the hypercube graph. Then the reflection of $\omega$ with respect to $i$ is given by
\begin{figure}
\centering
\begin{subfigure}{1.0\textwidth}
  \centering
  \includegraphics[width=0.8\linewidth]{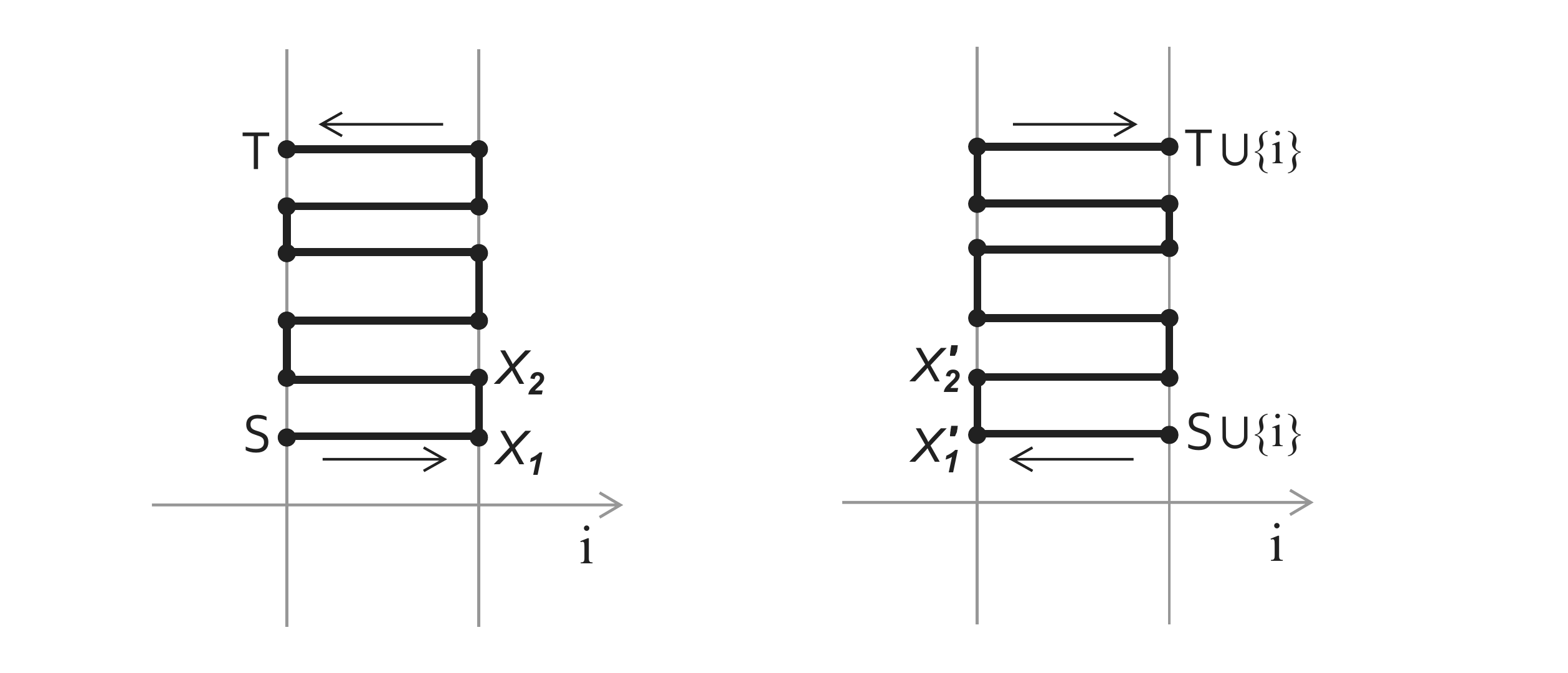}
\end{subfigure}%

\caption{A coalition path from $S$ to $T$ and its reflection w.r.t. $i$.}
\label{figure3}
\end{figure}
\be
\omega': X'_0 \to X'_1 \to \dots \to X'_\tau \nn
\ee
where $X'_t := X_t \cup \{i\}$ if $i \notin X_t$, and $X'_t := X_t \setminus \{i\}$ if $i \in X_t$. We observe that the total contribution of the player $i$ (that is, the sum of $\p_i v$'s) along the paths $\omega$ and $\omega'$ has the opposite sign, because whenever the player $i$ joins or leaves coalition along $\omega$, $i$ leaves or joins coalition along $\omega'$. On the other hand, we will show later that integrating $\p_i v$ over all paths traveling from $S$ to $T$ yields $\Phi_i(v, T) - \Phi_i(v, S)$ (see \ref{transition}), while integrating over all reflected paths yields $\Phi_i(v, T\cup \{i\}) - \Phi_i(v, S \cup \{i\})$. With this, the opposite sign of path integrals now yields  the reflection property, 
 which provides information about the values at two different states $S,T$ in terms of their relationship with $S \cup \{i\}, T \cup \{i\}$. This allows us to determine the allocation $\Phi$ for all games $v$ and coalitions $S \subset N$. 
 
\begin{theorem}\label{main}
There exists a unique allocation map $\Phi = (\Phi_i)_{i \in N} :  \cG_N \times 2^N \to \R^{n}$ that satisfies {\bf A1--A5} for all $N$ and $\Phi(v, \emptyset) = {\bf 0}$. The map $\Phi$ is represented by \eqref{value}.
\end{theorem}

In summary, while the Shapley formula \eqref{eqn:shapleyPermutation} considers coalition processes solely in the joining direction, the path integral formula \eqref{value} enables coalitions to proceed in either direction, thereby determining the value $\Phi(S)$ for all $S \subset N$ and $i \in N$. To the best of my knowledge, the theorem constitutes the first such simultaneous characterization.

\begin{remark} The following condition appears weaker than A5.

{\rm {\bf A5'}(reflection):} For any $v \in \cG_N$ and $S \subset [N] \setminus \{i,j\}$ with $i \neq j$, it holds
\be 
 \Phi _i [v]( S \cup \{ i,j \}) - \Phi _i [v]( S \cup \{ i \}) = -\big( \Phi _i [v](S \cup \{j\}) - \Phi _i [v](S) \big).
 \ee
However, A5 and A5' are equivalent since repeated application of A5' clearly yields A5. One may be tempted to interpret A5' as ``$j$ joining the coalition when $i$ is not a member makes $i$ worse off by the exact amount that $j$ joining the coalition when $i$ is a member makes $i$ better off." While this interpretation appears plausible, it appears to attribute the change of value $\Phi _i [v](S \cup \{j\}) - \Phi _i [v](S)$ solely to $j$. Even if there is only one step for the transition from $S$ to $S \cup \{j\}$, there still exists infinitely many paths between them, making it difficult to attribute the value change solely to $j$. In addition, we recall that the value $\Phi _i [v](S)$ makes a premise that $S$ is the terminal coalition. If the value $-\big(\Phi _i [v](S \cup \{j\}) - \Phi _i [v](S) \big)$ is interpreted as ``$j$ joining the coalition when $i$ is not a member makes $i$ worse off...", it appears that the precondition is forgotten. For these reasons, we refrain from providing normative justification for A1--A5, nor do we refer to them as ``axioms". These properties are the characteristics that the value $\Phi$ exhibits, and Theorem \ref{main} shows that they determine the value for all coalitions simultaneously.\footnote{This theorem could be better understood as a contribution to computer science, as it provides a finite set of properties that collectively determine the solution to the combinatorial Poisson's equation on graph (see Section \ref{AvePoisson}), which appears to be the first of its kind.}
\end{remark}

\begin{remark}\label{bettersetup}
Theorem \ref{main} appears to heavily depend on the symmetry of the graph \eqref{oldG} and the transition probability \eqref{MC}, particularly through A5. \eqref{MC} may be viewed as inappropriate for some situations. For example, in glove game, \eqref{MC} says with probability 1/3 the state $\{1,2\}$ may transition to $\{1\}$, for which the marginal value of player 2 is $-1$. A criticism could be that player 2 appears to pay a dissolution fee to player 1 for dissolving the $\{1,2\}$ coalition, but then player 2 would not willingly leave the $\{1,2\}$ coalition. To respond, we recall that \eqref{MC}, if adopted by the value allocator (VA), is VA's best likelihood estimate for coalition progression given information to VA. For example, VA may be aware of an outside incentive for player 2 to dissolve $\{1,2\}$. If VA realizes that the transition from $\{1,2\}$ to $\{1\}$ is truly impossible, it will assign the corresponding transition probability as 0. Alternatively, VA may perform a new marginal value assessment for each transition other than $\p_i v$ in \eqref{partialdifferential}. We recognize that \eqref{MC} for transition estimation and \eqref{partialdifferential} for marginal value assessment may not be suitable for all coalition game scenarios.  In the following sections, we will discuss various generalizations. However, for most general setups, we note that  the analogous value characterization result might not be available. In this paper,   we will be content to showcase through Theorem \ref{main}  that, in some cases, such characterization is possible.
\end{remark}

\begin{remark}[Characterization of the Shapley value via potential function approach] \label{moreaxioms} 

Let $v_\es \in \cG_\es$ denote the null game; $v_\es (\es) = 0$.  \citet{hart1989potential} defines a potential function $P : \cG \to \R$ as satisfying $P(v_\es) = 0$ and the following condition
\be\label{HaMa1}
\sum_{i \in N} D^i P(v) = v(N) \ \text{ for all } v \in \cG_N,
\ee
where $D^i P(v) := P(v) - P(v_{-i})$ is called the marginal contribution of a player $i$ in a game $v$. The following is a potential-based characterization of the Shapley value.
\begin{theorem}[\citet{hart1989potential}]\label{HaMa0}
There exists a unique potential function $P$. For any $v \in \cG_N$, the payoff $\big(D^i P(v)\big)_{i \in N}$ coincides with the Shapley value of $v$.
\end{theorem}
Noting that the formula \eqref{HaMa1} can be rewritten as
\be\label{HaMa2}
P(v) = \frac{1}{n} \bigg[ v(N) + \sum_{i\in N} P(v_{-i})\bigg],
\ee
we see that $P(v)$ is uniquely determined by the values of $P$ for the subgames of $v$. \citet{hart1989potential} also provides an explicit formula for the potential as 
\be\label{HaMa3}
P(v) = \E\bigg[\frac{n}{|S|} v(S)\bigg]
\ee
where the expectation is taken with respect to a probability distribution $p$ over $2^N$, defined as $p(S) :=1/n \cdot 1/ {n \choose s} =  \frac{s! (n-s)!}{n!n}$, where $s = |S|$, $n = |N|$. This allows one to interpret the potential as the expected normalized worth.

According to Hart and Mas-Colell, Theorem \ref{HaMa0} can be viewed as a characterization of the Shapley value based solely on one axiom \eqref{HaMa1}. They also note that 
$P$ serves as a formal mathematical potential, with the Shapley value as its gradient.

Concerning the allocation of value of a game $v \in \cG_N$ for partial coalitions $T \subsetneq N$, the allocation scheme proposed by Hart and Mas-Colell based on potential function differences is proven to be equivalent to the extended Shapley value. Consequently, it faces the same limitation of not fully capturing the entire game structure. In contrast, $\Phi$  thoroughly explores the game structure by integrating players' marginal value along each general coalition path. This leads to an allocation of the collaborative reward $v(T)$ that differs from the Shapley value, as illustrated in \eqref{glovegameextended} for example.
\end{remark}

\section{Hodge-Shapley value}\label{valuegeneralized}

The null-player axiom, which states that a player who contributes no marginal value to any coalition receives nothing, is crucial in determining the Shapley value.

It is worth noting that the Shapley value \eqref{eqn:shapleyPermutation} for player $i$ can be rewritten as
\begin{align}
  \label{eqn:shapleyPermutation1}
  \phi _i (v) = \frac{ 1 }{ n ! } \sum _{\sigma} \bigg[ \sum_{j \in N} \p_i v \big(S^{\sigma}_j, S^{\sigma}_j \cup \{ j \}  \big) \bigg],
    \end{align}
because the only nonzero term in the path-integral, the sum over $j$ in the bracket, is when $j = i$. This indicates that the role of the coalition game $v$ is simply yielding the marginal contribution of player $i$ as the form $\p_i v$. Although opting for player $i$'s marginal value in this manner might initially appear reasonable, particularly given Shapley's null player axiom, we now contend that it represents just one of several possibilities. The only essential attribute $\p_i v$ may have is that it belongs to the space of edge flows
\be\label{edgeflow}
\ell^2(\cE) = \big\{ f : \overline \cE \to \R \, \big| \, f \bigl(S \cup \{ i \}, S \bigr) = - f \bigl( S , S \cup \{ i\} \bigr) \text{ for all }  i \in N,\, S \subset N \setminus \{i\} \big\}.
\ee
An edge flow $f$ is thus a function on the edges satisfying the sign-changing property. Now we shall define the player $i$'s marginal value as an arbitrary edge flow $f_i \in \ell^2(\cE)$, and contend that this is a practically relevant generalization because, for example, even when only some of the players make progress at a given cooperative stage, the reward is usually distributed to all players in the cooperation in practice.

When player $i$'s marginal value is given by $f_i \in \ell^2(\cE)$, a natural generalization of the Shapley value can now be given by replacing $\p_i v$ with $f_i$ in  \eqref{eqn:shapleyPermutation1}, yielding
\begin{align}
  \label{eqn:shapleyPermutation2}
  \phi _{f_i} := \frac{ 1 }{ n ! } \sum _{\sigma} \sum_{j \in N} f_i \big(S^{\sigma}_j, S^{\sigma}_j \cup \{ j \}  \big).
  \end{align}
Note that $f_i \big(S^{\sigma}_j, S^{\sigma}_j \cup \{ j \}  \big)$ is not necessarily zero even when $j \neq i$. Furthermore, the coalition game $v $ is no longer present in the formula.

The average path integral formula \eqref{value} also naturally generalizes as
\be\label{value1}
\Phi_{f_i} (T) = \int_\Omega  \sum_{t=1}^{\tau_{T}(\omega)} f_i \big(X_{t-1}(\omega), X_t(\omega) \big) d \cP(\omega) = \E\bigg[ \sum_{t=1}^{\tau_{T}} f_i \big(X_{t-1}, X_t \big)\bigg],
\ee
again allowing us to interpret $\Phi_{f_i}(T)$ as the player's expected total contribution, thus her fair share, given that the coalition state advances from $\emptyset$ to $T$ and the player $i$'s marginal contribution for each transition is now given by an  edge flow $ f_i$.

Let $\mathbf{f}= (f_1,...,f_n)$ denote the marginal values of players, $\phi_{\bf f} := (\phi_{f_i})_{i \in N}$, and let $\Phi_{\bf f} := (\Phi_{f_i})_{i \in N} : 2^N \to \R^{n}$ define the value allocator associated to $\bf f$. 
\begin{theorem}\label{coincide2}
Let $G= (\cV, \cE)$ be the coalition graph and ${\bf f} \in \big(\ell^2(\cE)\big)^{n}$ denote the marginal values of players. Then $  \phi _{\bf f} = \Phi_{\bf f} (N)$.
   \end{theorem}
   
 \begin{proof}
Fix $i \in N$ and $S \subset N$ with $|S| \le n - 1$. Observe the map $f_i \in \ell^2(\cE) \mapsto \phi_{f_i}$ is linear. By linearity, it is enough to prove the proposition when $f_i  =  \chi _{ ( S, S \cup \{ k \} ) }$ for any fixed $k \in N \setminus S$, where $\chi _{ ( S, S \cup \{ k \} )}$ is the indicator equal to $1$ at $( S, S \cup \{ k \} ) $ (thus $-1$ at $(S \cup \{ k \}, S ) $), and $0$ on all other edges in $\overline \cE$. First, observe \eqref{eqn:shapleyPermutation2} yields
\be\label{good1}
\phi_{ \chi _{ ( S, S \cup \{ k \} ) }} =  \frac{ |S|! ( n - |S| - 1) ! }{ n ! }.
\ee
Next, define an edge flow $\displaystyle g := \sum _{ \lvert T \rvert = |S|,\, j \notin T }  \chi  _{ ( T, T \cup \{ j \} ) } $, and observe
\begin{align}
\Phi_g (N) &= \sum _{\lvert T \rvert = |S|,\, j \notin T }  \Phi_{\chi  _{ ( T, T \cup \{ j \} )}} (N) \\
&= \binom{n }{|S|} \bigl( n - |S| \bigr) \Phi_{ \chi _{ ( S, S \cup \{ k \} ) }}(N) \nn
\end{align}
where the second equality is by the fact $\Phi_{\chi  _{ ( T, T \cup \{ j \} )}} (N) = \Phi_{\chi  _{ ( S, S \cup \{ k \} )}} (N)$  due to the symmetry of the coalition graph \eqref{oldG} and  the random coalition process \eqref{MC}. 

Noting that $\binom{n }{|S|} \bigl( n - |S| \bigr) = \frac{ n ! }{ |S| ! (n - |S| - 1) ! }$, it remains to show $\Phi_g (N)=1$. By the definition of the flow $g$, it is clear that for any sample path $\omega$ of the coalition process, we have the following pathwise equality (recalling $X_0= \emptyset$)
\be
 \sum_{t=1}^{\tau_{N}(\omega)} g \big(X_{t-1}(\omega), X_t(\omega) \big) =1. \nn
\ee
Hence 
$\Phi_g (N) = \E[ \sum_{t=1}^{\tau_{N}} g (X_{t-1}, X_t )] = 1$, proving the theorem.
\end{proof}
Theorem \ref{coincide1} becomes a special case when $f_i = \p_i v$. In light of the theorem, we may call $ \phi_{\bf f}$ the Hodge-Shapley value associated to ${\bf f}$, 
with $ \Phi_{\bf f}$ its extension to all coalitions. 

An example of less strict marginal value assignment may be given as follows. Given $\al \in \R$ and a game $v \in \cG_N$, we define player $i$'s marginal value by (cf. \eqref{partialdifferential})
\begin{equation}\label{modifiedmarginalvalue}
f_{\al, i} \bigl( S , S \cup \{ j \} \bigr) :=
    \begin{cases}
\al \big( v(S \cup \{ i \}) - v(S) \big) & \text{if } \ j=i, \\
  \frac{ (1 -  \al)}{n -1} \big( v(S \cup \{ j \}) - v(S) \big) & \text{if } \  j \neq i.
    \end{cases}
  \end{equation} 
Notice $f_{\al, i} = \p_i v$ if $\al=1$. For $\al \in (0,1)$, the marginal value \eqref{modifiedmarginalvalue} 
is such that for the transition from $S$ to $S \cup \{i\}$, player $i$ receives the $\al$ proportion of the marginal value $ v(S \cup \{ i \}) - v(S)$, and $(1-\al) \big(v(S \cup \{ i \}) - v(S) \big)$ is equally distributed to the rest of the players. We may call $\phi_{{\bf f_\al}} =( \phi _{f_{\al, i}})_{i \in N} $ the $\al$-Shapley value, with $ \Phi_{{\bf f_\al}} = (\Phi_{f_{\al,i}})_{i \in N} : 2^N \to \R^{n}$ its extension. The null-player axiom may not hold for the $\al$-Shapley value.

\begin{example} Let $N=2$, and $v \in \cG_{\{1,2\}}$ be given by $
v (\emptyset) = v(\{2\}) = 0$, $ v (\{1\}) = v(\{1,2\}) = 1$. Note that $\p_2 v = 0$, thus the Shapley value $\phi_2 (v) = 0$ for player 2. On the other hand, the $\al$-Shapley value for player 2 can be easily calculated as $1 - \al$. Player 2 continues to receive the $1-\al$ portion of the grand coalition value.
\end{example}

\begin{example}\label{glove2}
We revisit the glove game from Example \ref{ex:introGlove} and calculate the $\al$-Shapley value  \eqref{eqn:shapleyPermutation2}. 
For this, we need to calculate (recall $N = \{1,2,3\}$)
\begin{align}
  \label{modifiedshapley}
\phi_{\al, i} (v) = \frac{ 1 }{ 6 } \sum _{\sigma} \big[ f_{\al, i} (\emptyset, \{\sigma(1)\}) + f_{\al, i}(\{\sigma(1)\}, \{\sigma(1), \sigma(2)\} )+ f_{\al, i} (\{\sigma(1), \sigma(2)\} , N ) \big], \nn
  \end{align}
  where $ f_{\al, i} (\emptyset, \{\sigma(1)\}) + f_{\al, i} (\{\sigma(1)\}, \{\sigma(1), \sigma(2)\} )+ f_{\al, i} (\{\sigma(1), \sigma(2)\} , N )$ represents the total contribution of player $i$ along the coalition path $\sigma$.
For example, if $\sigma = (1,2,3)$ (that is, the player $1$ joins first, followed by the players $2$ and $3$), this sum equals $\frac{1-\al}{2}$ for $i=1,3$ and $\al$ for $i=2$, because a pair of gloves is made precisely when the player $2$ joins in this path. Thus, player $1$ contributes marginal value $\frac{1-\al}{2}$ when joining the coalition first (2 of 6 permutations) and marginal value $\al$ otherwise (4 of 6 permutations), so $\phi _{\al, 1} (v) = \frac{ 1 + 3\al }{ 6 } $. Similarly, $\phi _{\al, 2} (v) =  \phi _{\al,3} (v) = \frac{ 5-3\al }{ 12 }$. This allocation coincides with the Shapley value if $\al=1$, and player $1$ receives more than players $2,3$ if and only if $\al > \frac13$.

In Section \ref{AvePoisson}, we will explain how to calculate the value \eqref{value1} in general and, as an example, obtain the following extended allocation table for the $\al$-Shapley value:
\be\label{alphashapley}
\begin{tabularx}{0.913\textwidth}
{ 
  | >{\centering\arraybackslash}X 
    | >{\centering\arraybackslash}X 
  | >{\centering\arraybackslash}X 
  | >{\centering\arraybackslash}X 
    | >{\centering\arraybackslash}X 
    | >{\centering\arraybackslash}X 
  | >{\centering\arraybackslash}X 
  | >{\centering\arraybackslash}X 
  | }
\hline 
 & $\{1\}$ & $\{2\}$ & $\{3\}$ & $\{1,2\}$ & $\{1,3\}$ & $\{2,3\}$ & $\{1,2,3\}$  \\
\hline 
$\Phi_{f_{\al, 1}}$ & $\frac{15 \al - 5 }{ 24} $ & $ \frac{5 - 15 \al }{ 48}$ & $ \frac{5 - 15 \al }{ 48}$ & $ \frac{ 7\al + 3}{16 }$ &$  \frac{ 7\al + 3}{16 }$ &$ \frac{ 1 - 3\al}{ 8}$ & $\frac{3\al + 1}{ 6}$  \\
\hline 
$\Phi_{f_{\al, 2}}$ &$\frac{5 - 15 \al }{ 48}$ & $\frac{3\al - 1 }{12 } $ & $ \frac{3\al - 1}{48 } $ & $  \frac{ \al + 5}{ 16 }  $ & $\frac{ 1 - \al}{ 2} $ & $\frac{3\al -1 }{16 } $ & $\frac{5 - 3\al }{ 12} $ \\
\hline
$\Phi_{f_{\al, 3}}$ & $\frac{5 - 15 \al }{ 48} $ & $ \frac{3\al - 1 }{ 48} $ & $\frac{ 3\al - 1}{ 12}$ & $ \frac{1 - \al }{2 }$ & $ \frac{ \al + 5}{16 }$ & $\frac{3\al -1 }{16 } $ & $\frac{5 -3\al}{ 12} $ \\
\hline
\end{tabularx}
\ee
We see that $\Phi_{f_{\al,i}}(N) = \phi_{\al, i}(v)$; the extended allocation at the grand coalition coincides with the $\al$-Shapley value, as claimed in Theorem \ref{coincide2}. 

\end{example}

\begin{remark}[Population Monotonic Allocation Schemes] 

\citet{sprumont1990population} argues that to address the possibility of partial cooperation effectively, it is necessary not only to determine the allocation of the grand coalition value $v(N)$ but also to determine how to allocate the value of every coalition $S$ in case players do not fully cooperate and $S$ is eventually formed. His concern is to ensure that once a coalition $S$ has agreed upon an allocation of $v(S)$, no player will be enticed to form a smaller coalition than $S$ through bargaining or other means. This requirement translates into the necessity for each player's payoff to increase as the coalition to which they belong grows larger. Consequently, Sprumont seeks a Population Monotonic Allocation Scheme (PMAS).
\begin{definition} A vector $x = (x_{iS})_{i \in S, S \subset N}$ is a population monotonic allocation scheme ({\rm PMAS}) of the game $v \in \cG_N$ if $x_{i\es} = 0$ for all $ i\in N$ and $x$ satisfies the following:
\begin{align}
\sum_{i \in S} x_{iS} = v(S) \,\text{ for all }\, S \subset N, \ \text{ and } \ 
x_{iS} \le x_{iT} \,\text{ for all }\, i \in S \subset T. \nn
\end{align}
\end{definition}
Sprumont then verifies sufficient conditions for a game $v$ to possess a PMAS, which includes the classes of quasiconvex games\footnote{$v$ is quasiconvex if $\sum_{i \in S} \big(v(S) - v(S \setminus \{i\})\big) \le \sum_{i \in S} \big(v(T) - v(T \setminus \{i\})\big)$ for all $ S \subset T$.} and Increasing Average Marginal Contributions (IAMC) games\footnote{$v$ is an IAMC game if $\frac{1}{|S|} \sum_{i \in S} \big(v(S) - v(S \setminus \{i\}) \big) \le \frac{1}{|T|} \sum_{i \in T} \big(v(T) - v(T \setminus \{i\})\big)$ for all $S \subset T$.
}, as well as the following equivalent condition.
\begin{theorem}[\citet{sprumont1990population}]
A game has a {\rm PMAS} if and only if it is the sum of a positive linear combination of monotonic simple games with veto control \footnote{$v$ is monotonic if $v(S) \le v(T)$ for all $S \subset T$. $v$ is a simple game if $v(S) \in \{0,1\}$ for all $S \subset N$. $i \in N$ is a veto player in a game $v$ if $v(S) = 0$ for all $S \subset N \setminus \{i\}$. A game with veto control is a game with at least one veto player.} and an additive game $v$, given by $v_a(S) := a |S|$ for all $S \subset N$, where $a \in \R$ is arbitrary. 
\end{theorem}
Sprumont also shows that, while every game with a PMAS is totally balanced, not every totally balanced game has a PMAS. For example, every assignment game is totally balanced, but many of them lack a PMAS.  \citet{abe2023core} shows that the Shapley value exists within the core of a game $v$ if and only if $v$ can be expressed as the summation of a linear combination of basic games meeting specific criteria. Geometrically, these basic games correspond to the extreme edges of a polyhedral structure. 

In summary, PMAS serves as a multivalued solution concept aimed at determining how to distribute the value of each coalition $S$. In contrast, our allocation scheme $\Phi$ in \eqref{value1} is single-valued. Investigating the characteristics of ${\bf f}$ that result in an allocation satisfying reasonable economic properties remains an intriguing avenue for future research. For instance, imposing ${\bf f} \ge {\bf 0}$ on $\cE$ (thus ${\bf f} \le {\bf 0}$ on $\cE_-$) seems pertinent to Sprumont's consideration (though not equivalent), since the condition implies no player possesses an immediate incentive to depart from the existing coalition.

\end{remark}

\begin{remark}[The egalitarian Shapley values] Given $v \in \cG_N$, let ${\rm ED}_i(v) := \frac{v(N)}{n}$ denote the equal division value. For $\al \in \R$ and $i \in N$, the convex combination of the Shapley value and the equal division value
\be\label{egal}
{\rm ES}^\al_i (v) := \al \phi_i(v) + (1-\al) {\rm ED}_i(v)
\ee
is called an egalitarian Shapley value (\citet{joosten1996dynamics}) for $0 \le \al \le 1$. Egalitarian Shapley values meet efficiency, symmetry, and linearity, but they do not meet the null player property if $\al \neq 1$. However, they satisfy the following weaker property
\vspace{1mm}

\noin\textbf{Null player in a productive environment} ($\mathsf{NPE}$): If $v(N) \ge 0$ and $i$ is a null player, i.e., $\p_i v \equiv 0$, then player $i$'s payoff $\varphi_i(v)$ is nonnegative.
\vspace{1mm}

$\mathsf{NPE}$ states that when the entire society is productive ($v(N)\ge 0$), null players should not receive negative payoffs (because they do not cause harm to the society). \citet{casajus2013null} claims that if we no longer require null players' payoffs to be zero, the nature of solidarity emerges in how null players are treated. Since null players are completely unproductive, their ``selfish'' payoffs, or Shapley payoffs, are zero. As a result, any nonzero payoff must be due to solidarity among the players. To prevent players from receiving unreasonable payoffs corresponding to the case $\al < 0$ in \eqref{egal}, Casajus and Huettner invokes the following axiom.
\vspace{1mm}

\noin{\bf Desirability:} If $v(S \cup \{i\}) \ge v(S \cup \{j\})$ for all $S \subset N \setminus \{i,j\}$, then $\varphi_i(v) \ge \varphi_j(v)$.
\vspace{1mm}

Desirability compares two players in a game to ensure that their payoffs do not contradict their productivities as measured by marginal contributions. 

A characterization of the class of egalitarian Shapley values is now given as follows.
\begin{theorem}[\citet{casajus2013null}] A value $\vp$ satisfies  efficiency, linearity, desirability, and the null player in a productive environment property if and only if there exists an $\al \in [0,1]$ such that $\vp = {\rm ES}^\al$.
\end{theorem}
Our definition of the marginal value \eqref{modifiedmarginalvalue} clearly indicates that the $\al$-Shapley value is related to the egalitarian Shapley value. To see this, we compute
\begin{align*}
\phi_{f_{\al,i}} &= \frac{ 1 }{ n ! } \sum _{\sigma} \sum_{j \in N} f_{\al, i} \big(S^{\sigma}_j, S^{\sigma}_j \cup \{ j \}  \big) \\
&= \frac{ 1 }{ n ! }\bigg[ \al  \sum _\sigma \Bigl( v \bigl( S^{\sigma}_i \cup \{ i \} \bigr) - v (S^{\sigma}_i ) \Bigr) +  \frac{ 1 -  \al }{n -1} \sum_{j \neq i} \sum_{\sigma} \Bigl( v \bigl( S^{\sigma}_j \cup \{ j \} \bigr) - v (S^{\sigma}_j) \Bigr)\bigg] \\
&= \al \phi_i(v) + \frac{ 1 -  \al}{n -1} \sum_{j \neq i} \phi_j(v)\\
&= \al \phi_i(v) + \frac{ 1 -  \al }{n -1} \big( v(N) - \phi_i(v) \big) \\
&= \frac{ n \al - 1}{n -1} \phi_i(v) + \frac{n(1-\al)}{n-1}\frac{v(N)}{n}\\
& = {\rm ES}^{\al'}_i(v), \ \text{ where }\, \al' = \frac{ n \al - 1}{n -1}.
\end{align*}
This demonstrates that the $\alpha$-Shapley value aligns with the $\alpha'$-egalitarian Shapley value. Consequently, $(\Phi_{f_{\alpha,i}})_{i \in N}$ serves as a microfoundation for the egalitarian Shapley value, and extends it to all coalitions $T \subset N$, as indicated by Theorem \ref{coincide2}.

\end{remark}

\section{General cooperative networks}\label{Hodgeallocation1}

We have extended Shapley's cooperative value allocation theory by considering the coalition game graph \eqref{oldG}, generalized marginal values represented by edge flows \eqref{edgeflow}, and the allocation operator as the average path integral of flows \eqref{value1}. We now observe that these concepts can be extended to more general network structures. 
This leads us to consider a general cooperative game graph $G=(\cV, \cE)$, which is a general connected graph with $\cV$ a finite set of cooperative states and $\cE$ a set of edges. Each $S \in \cV$ is not necessarily a subset of $N$, but $\cV$ now represents an arbitrary finite set, with each  $S \in \cV$ describing a general cooperative situation. For example,  $S,T \in \cV$ 
may both represent cooperations among the same group of players but working under different conditions.

As an example of the situation of interest, we may consider the following: Let $\cV$ represent the set of all possible states of a given project, in which the project manager, or principal, wishes to reach the project completion state $F \in \cV$. The project state can move from $S$ to $T$ with a certain probability, if there is an edge between them. For the project's advancement, the manager hires $N$ agents, or employees. 

We now ask the following question: Given the principal's reward function in each state and her payoff function to agents at each state transition, what is her expected revenue when the project is completed, and what is her expected liability to each agent?

The question leads us to consider a general cooperative process $(X_t)_{t \in \N_0}$ valued on the state space $\cV$ with a given initial state $X_0 = O \in \cV$. Let $(\Omega, \cF, \cP)$ represent the probability space. In general, the governing law $\cP$ of the process $(X_t)_{t \in \N_0}$ can be arbitrary, except for the requirement that $X_{t+1}$ be one of the adjacent states of $X_t$.

\begin{figure}
\centering
\begin{subfigure}{.9\textwidth}
  \centering
  \includegraphics[width=0.7\linewidth]{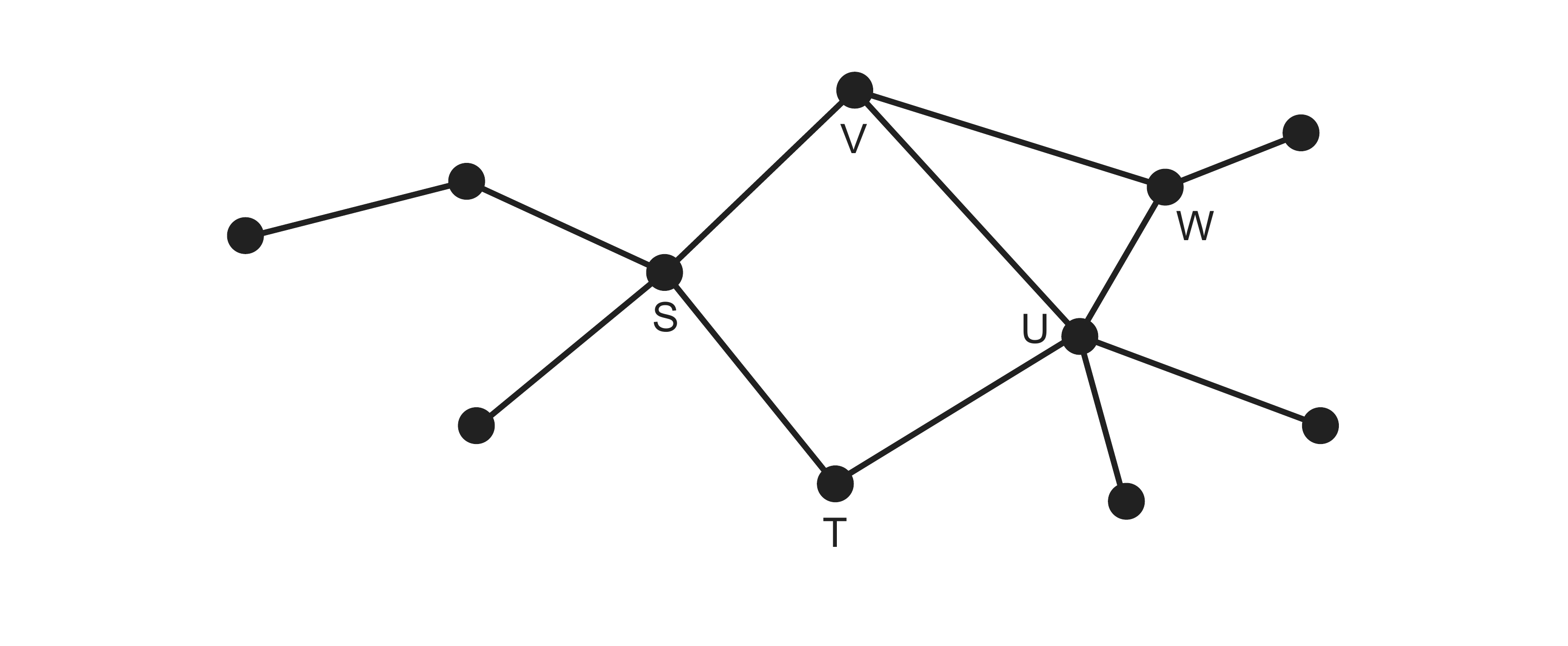}
\end{subfigure}%

\caption{A cooperative graph can have an arbitrary structure.\\ Each edge has a forward and reverse direction associated to it.}
\label{figure4}
\end{figure}
Let $(S,T) \in \cE$ denote a forward edge directed from $S$ to $T$, with its reverse $(T,S) \in \cE_-$. Set $\overline \cE := \cE \cup \cE_-$.\footnote{Thus $\cE \cap \cE_- = \emptyset$, and for each $S \neq T$ in $\cV$, either $(S,T) \in \cE$, or $(S,T) \in \cE_-$, or $(S,T) \notin \overline\cE$.} The above requirement now reads $(X_t, X_{t+1}) \in \overline \cE$. 

 Let $\ell^2(\cE)$ denote the set of all edge flows $f: \overline\cE \to \R$ satisfying the sign changing property $f(T,S) = - f(S,T)$. Motivated from Section \ref{valuegeneralized}, we continue to assume that each agent $i \in N$ is associated with an edge flow $f_i \in \ell^2(\cE)$, which represents the agent $i$'s marginal value. We then define the path integral 
\be\label{pathintegral2}
{\cal I}_{f_i}(T) = {\cal I}_{f_i}(T)(\omega) := \sum_{t=1}^{\tau_{T}(\omega)} f_i  \big(X_{t-1}(\omega), X_t(\omega) \big)
\ee
where $\tau_{T} = \tau_{T}(\omega) \in \N$ denotes the first  time the process $\big(X_t(\omega)\big)_{t \in \N_0}$ visits $T$. Given that the project state has progressed from $O$ to $T$ along the path $\omega$, \eqref{pathintegral2} represents the agent $i$'s total contribution throughout the progression. The value allocation function can then be analogously defined as the average path integral
\be\label{value2}
\Phi_{f_i} (T) := \int_\Omega  {\cal I}_{f_i}(T)(\omega)  d \cP(\omega) = \E[ {\cal I}_{f_i}(T)] \ \text{ for each } T \in \cV.
\ee
$\Phi_{f_i} (T)$ represents the agent's expected total contribution given the  state advances from $O$ to $T$, where $f_i$ represents the agent $i$'s marginal contribution for each transition.

\begin{remark}[Initial state]

In coalition games, the vacant coalition $\emptyset$ usually serves as the starting point for coalition progression. However, in a general cooperative network, any state may serve as the starting point. We use ${\cal I}^S_{f_i}(T)$ and $\Phi^S_{f_i}(T)$ to represent the path integral and its expected value similar to \eqref{pathintegral2} and \eqref{value2}, with the distinction that the superscript denotes the initial state, i.e. $X_0 = S$, while $T$ denotes the terminal state as before. We will drop the superscript when $X_0 = O$. 

Importantly, the identity below will be useful and proven later
\be\label{transition}
 \Phi_f^S(T) +  \Phi^T_f(U) =  \Phi^S_f(U)
 \ee
for any $f \in \ell^2(\cE)$ and $S,T,U \in \cV$. In particular, $ \Phi_f^T(T) = 0$ (by setting $U=T$), and $ \Phi_f^S(T) = -  \Phi_f^T(S)$ (by setting $U=S$). Note that given a fixed initial state $O$, this allows us to represent $\Phi_f^S$ in terms of $\Phi_f$ as $\Phi^S_f(T) = \Phi_f(T) - \Phi_f(S)$.
\end{remark}

We can now provide a general answer to the question. Let $\cV$ represent the project state space in which the manager wishes to achieve the project completion state $F \in \cV$. Let $v : \cV \to \R$ denote the manager's revenue, i.e., $v(U)$  represents the manager's final revenue if the project ends at the state $U$. Let $N = \{1,...,n\}$ denote the employees with their marginal contributions $f_1,...,f_n \in \ell^2(\cE)$. Because it is her contribution and share, the manager must pay $f_i (S,T)$ to the employee $i$ at each state transition from $S$ to $T$. Thus, the manager's surplus in this single transition is given by $v(T) - v(S) - \sum_i f_i (S,T)$. Now the manager's revenue problem is: What is the manager's expected revenue if they begin at the initial project state $O$ (with $v(O)=0$), and the manager's goal is to reach the project completion state $F$?

We can observe that the answer is $v(F) - \sum_i \Phi_{f_i}(F)$, where $\Phi_{f_i}$ is given by \eqref{value2}. (So if this is negative, the manager may decide not to begin the project at all.)

Furthermore, in the middle of the project, the manager may want to recalculate her expected gain or loss. That is, suppose the current project status is $T$, and they arrived at $T$ via a specific path $\omega$, and thus the manager has paid the payoffs, i.e., the path integrals \eqref{pathintegral2}, to the employees. The manager may wish to recalculate the expected gain if she decides to proceed from $T$ to $F$. This is now provided by
\[
v(F) - v(T) - \sum_i \Phi^T_{f_i} (F) = v(F) - v(T) -  \sum_i \big(\Phi_{f_i} (F) - \Phi_{f_i} (T) \big),
\]
and the manager can make decisions based on the expected revenue information. 

\begin{remark}[Efficiency]\label{remark:efficiency}
Shapley's efficiency axiom is a crucial ingredient for characterizing his value allocation scheme; without it, it is difficult to establish the uniqueness of the allocation. The efficiency axiom is equally important for our charaterization result (Theorem \ref{main}) to yield a unique allocation for all coalitions.

Our description of the principal-agent allocation problem, on the other hand, shows that efficiency is merely a constraint, which is equivalent to declaring that the principal's marginal surplus $v(T) - v(S) - \sum_i f_i (S,T)$ is identically zero for every $(S,T) \in \overline\cE$. The principal enters the problem as soon as we relax this vanishing constraint, and $v: \cV \to \R$ then represents her value at each cooperative state. The crucial difference between principal and agents is that the value of the principal is represented as a function $v$ on $\cV$, whereas the (marginal) value of the agents is represented as edge flows $(f_i)_{i \in N}$ on $\overline\cE$. Finally, we can recover the efficiency $\sum_i \Phi_{f_i}(T) = v(T)$ for all $T \in\cV$ by imposing the following marginal efficiency condition 
\be\label{marginalefficiency}
v (T) - v(S) = \sum_i f_i (S,T) \ \text{ for every } (S,T) \in \overline\cE.
\ee
For example, the $\al$-Shapley value \eqref{modifiedmarginalvalue} satisfies this condition, hence is efficient.
\end{remark}

\begin{remark}[Myerson's conference structures and fair allocation rules]\label{Meyerson}
\citet{myerson1977graphs, myerson1980conference} considers how the outcome of a cooperative game should be determined by which groups of players hold cooperative planning conferences. He suggests an allocation rule, which are functions that map conference structures to payoff allocations. The rule is characterized by a notion of fairness, which we briefly explain. 

Let $N=\{1,...,n\}$ be the set of players. Let $V : 2^N \to 2^{\R^n}$ be a (set valued) function such that i) $V(S)$ is a closed subset of $\R^n$, ii) $\es \neq V(S) \neq \R^n$ if $S \neq \es$ ($V(\es) = \R^n$), and iii) if $x \in V(S)$, $y \in \R^n$ and $ y_i \le x_i$ for all $i \in S$, then $y \in V(S)$. The set $V(S)$ is interpreted as the set of all payoff allocations that provide members of $S$ with a combination of payoffs that they can guarantee for themselves without cooperating with the other players. In \citet{myerson1980conference}, such a $V$ is referred to as a game.

For any $S \subset N$, let $\p V(S)$ be the weakly Pareto-efficient frontier of $V(S)$, i.e.,
\be
\p V(S) = \{x \in V(S) \, | \ \text{if } y_i > x_i \text{ for all } i \in S, \text{ then } y \notin V(S) \}. \nn
\ee
Any $S \subset N$ with $|S| \ge 2$ is called a conference. A conference structure is then any collection of conferences. Let CS denote the set of all possible conference structures
\be
{\rm CS} = \{Q \, | \, S \subset N \text{ and } |S| \ge 2 \text{ for all } S \in Q\}. \nn
\ee
Players $i$ and $ j$ are connected by $Q$ if $i=j$ or there exists a sequence of conferences $\{S_1,...,S_m\} \subset Q$ such that $i \in S_1$, $j \in S_m$ and $S_k \cap S_{k+1} \neq \es$ for all $k=1,...,m-1$. Then $N / Q$ denotes the partition of $N$ defined by this connectedness relation, i.e.,
\be
N/Q = \big\{ \{ j \, | \, i \text{ and } j \text{ are connected by } Q\} \, \big| \, i \in N \big\}. \nn
\ee
That is, the sets in $N/Q$ are the maximal connected coalitions determined by $Q$. 

An allocation rule for the game $V$ is any function $\Psi : {\rm CS} \to \R^n$ such that
\be
\Psi(Q) \in \p V(S) \ \text{ for every }\, Q \in {\rm CS},\, S \in N/Q.
\ee
This condition asserts that if $S$ is a maximally connected coalition for the conference structure $Q$, then its members should coordinate themselves to achieve a (weakly) Pareto-efficient allocation among the allocations available to them. 

For $i \in N$, $S \subset N$, the conference structures $Q - S$ and $Q -^* i$ are defined by
\be
Q - S = \{ T \, | \, T \in Q \text{ and } T \neq S\}, \q Q -^* i = \{T \, | \, T \in Q \text{ and } i \notin T\}. \nn
\ee
So $Q - S$ is the conference structure that differs from $Q$ in that $S$ is removed from the list of permissible conferences, whereas $Q -^* i$ is the conference structure that differs from $Q$ in that all conferences containing player $i$ are eliminated.

Myerson defines that an allocation rule $\Psi : {\rm CS} \to \R^n$ is fair if
\be\label{fair}
\Psi_i(Q) - \Psi_i (Q-S) = \Psi_j(Q) - \Psi_j (Q-S) \ \text{ for all } Q \in {\rm CS}, S \in Q, i \in S, j \in S.
\ee
$\Psi(\cdot)$ is fair if the members of $S$ decide not to meet together, the change in conference structure (from $Q$ to $Q-S$) should have an equal impact on all members of $S$. 

On the other hand, an allocation $\Psi : {\rm CS} \to \R^n$ has balanced contributions if
\be\label{balanced}
\Psi_i(Q) - \Psi_i (Q-^* j) = \Psi_j(Q) - \Psi_j (Q-^* i) \ \text{ for all } Q \in {\rm CS}, i \in S, j \in S.
\ee
$\Psi(\cdot)$ has balanced contributions if $j$'s contribution to $i$  equals $i$'s contribution to $j$ in any conference structure. Now a main result of \cite{myerson1980conference} is the following.
\begin{theorem}[\citet{myerson1980conference}]
There exists a unique fair allocation rule $\Psi(\cdot)$ for the game $V$. This allocation rule also has balanced contributions. Conversely, if $\Psi(\cdot)$ has balanced contributions, then $\Psi(\cdot)$ is a fair allocation rule.
\end{theorem}
The conference structure $Q$ appears relevant to our framework as it can be interpreted as a graph $\cG$, with each node being either a conference $ S \in Q$ or a singleton $\{i\}$ for each $\displaystyle{i \notin \cup_{S \in Q} S}$. Let two conferences $S$ and $T$ in $Q$ be connected by an edge if $S \cap T \ne \es$, while every singleton $\{i\}$ with $\displaystyle{i \notin \cup_{S \in Q} S}$ be isolated in $\cG$. Notice then each coalition in $N/Q$ corresponds to the union of all nodes in a maximal connected subgraph of $\cG$. In this viewpoint, Myerson's conference structures differ from the coalition graphs \eqref{oldG}, representing a more general cooperative networks (while its nodes continue to be restricted to coalitions).  Myerson's allocation rule $\Psi$ depends on both the game $V$ and the conference structure $Q$. This dependence arises from the assumption that players will form the largest feasible coalitions aligning with the conference structure. $\Psi$ is characterized by the fairness axiom, resulting in an implicit allocation. In contrast, our allocation $\Phi$ is determined by the cooperative network $\mathcal{G}=(\mathcal{V}, \mathcal{E})$, the law $\mathcal{P}$ governing the cooperative process $(X_t)_{t \in \mathbb{N}_0}$, the marginal values of players ${\bf f} = (f_i)_{i \in N}$, and the arbitrary initial and terminal cooperative states $O$ and $T$ in $\mathcal{V}$. As a result, $\Phi$ provides a flexible allocation framework, which is explicitly represented by the formula \eqref{value2}. 
\end{remark}

\section{Reversible Markov chains and Poisson's equations on graphs}\label{AvePoisson}

The preceding discussions naturally lead to the question of how to evaluate the value function \eqref{value2}, which represents an infinite sum of all possible paths between states. As the network gets more complicated, this can quickly become intractable.

In this section, we establish the relationship between the value function \eqref{value2} and the Poisson's equation on graphs, when the cooperative process follows an important class of probability laws, known as reversible Markov chains. This will allow us to compute \eqref{value2} by solving a system of linear equations, which is tractable.

To this end, we need to introduce basic linear operators, such as gradient, divergence, and Laplacian, between linear function spaces $\ell^2(\cV)$ and $\ell^2(\cE)$. We refer to \citet{lim2020hodge} for an accessible introduction to combinatorial Hodge theory. Recently, the combinatorial Hodge decomposition has found application in game theory across diverse contexts, including noncooperative games \citep{candogan2011flows}, cooperative games \citep{stern2019hodge}, and the ranking of social preferences \citep{jiang2011statistical}.

 Let $\ell^2(\cV)$ denote the space of functions  $ \cV \rightarrow \mathbb{R} $ with the standard inner product
\begin{equation}\label{inner1}
   \langle u , v \rangle \coloneqq \sum _{ S \in \cV }  u (S) v (S).
\end{equation}
We recall that these are called coalition games if $\cV = 2^N$ and $v(\emptyset) = 0$. In contrast, $\cV$ can now be an arbitrary finite set and $v \in \ell^2(\cV)$ is not required to assume $0$ anywhere. 

Let $\la : \overline{\cE} \to \R_+$ define the edge weight, satisfying $\la (T,S) = \la (S,T) \ge 0$ (i.e., no sign alternation) for all $S,T \in \cV$. We declare that there is an edge between $S$ and $T$, i.e., $(S,T) \in \overline{\cE}$, if and only if $\la (S,T) > 0$. Given an edge weight $\la$, we denote by $ \ell ^2_\la (\cE) $ the space of functions
$ \overline{\cE} \rightarrow \mathbb{R} $ equipped with the weighted inner product
\begin{equation}\label{inner2}
   \langle f , g \rangle_{\la} \coloneqq \sum _{ (S,T) \in \cE } \la (S,T) f (S,T) g (S,T).
\end{equation}
with the sign changing property $ f(T,S) = - f (S,T) $. Elements in $ \ell ^2_\la (\cE) $ are often called edge flows. Now the gradient operator $ \mathrm{d} \colon \ell ^2 (\cV) \rightarrow \ell_\la ^2 (\cE) $ is defined by
\begin{equation}\label{gradient}
  \mathrm{d} v ( S,T ) \coloneqq v (T) - v (S) \ \text{ for each } (S,T) \in \overline{\cE}.
\end{equation}
Given a game  $v$ on $\cV$, $ \mathrm{d} v$ measures its marginal value for each edge $(S,T) \in \overline\cE$. 

Let $ \mathrm{d}^* \colon \ell ^2_\la (\cE) \rightarrow \ell^2 (\cV) $ denote the adjoint of $ \mathrm{d}$. $ \mathrm{d}^*$ is called divergence operator, which is characterized by the defining relation for the adjoint operator
\be\label{adjointdefinition}
   \langle \mathrm{d}v, f \rangle_{\la} =    \langle v, \mathrm{d}^*f \rangle \ \text{ for every } v \in \ell^2(\cV) \text{ and } f \in \ell_\la^2(\cE).
\ee
It is important to note that $\mathrm{d}$ is defined by \eqref{gradient} and is independent of the edge weight $\la$. However, $ \mathrm{d}^*$ depends on the choice of $\la$ due to its defining relationship \eqref{adjointdefinition}, so it must be called a weighted divergence, but we simply call it divergence. In fact, \eqref{adjointdefinition} gives the following explicit form of the divergence
\begin{align}\label{div}
\mathrm{d}^*f (S) 
=  \sum_{T \sim S}\la (T,S) f(T,S) \ \text{ for every } S \in \cV,
\end{align}
where $ T \sim S $ means $\la(S,T) >0$, i.e., $S$ and $T$ are adjacent in the network. 

The Laplacian is the symmetric (self-adjoint) operator $ \mathrm{L} = \mathrm{d} ^\ast \mathrm{d} : \ell^2(\cV) \to \ell^2(\cV)$:
\begin{align*}
\mathrm{L} v (S) =  \sum_{T \sim S}\la (S,T) \big( v(S) - v(T) \big).
\end{align*}
$\mathrm{L} v(S)$ calculates the weighted sum of $v$'s marginal increment directed to each state $S$.

Poisson's equation is a partial differential equation of broad utility with the form $\mathrm{L}u = h$, where $\mathrm{L}$ represents the Laplace operator. Typically, a function $h$ is given, and $u$ is sought. Given a weighted graph $\cG= (\cV, \cE)$ with weight $\la$ and an edge flow $f \in \ell^2_\la (\cE)$, we are interested in the following form of the Poisson's equation
\be\label{Poisson3}
\mathrm{L} u =  \mathrm{d} ^\ast f. 
\ee
The solution $u$ can be interpreted as follows: given an edge flow $f$, the potential function $u : \cV \to \R$ solving \eqref{Poisson3} is such that its gradient flow $\mathrm{d}u : \overline{\cE} \to \R$ is closest to the flow $f$. To see this, we recall the fundamental theorem of linear algebra:
\be\label{Hodge}
\ell^2(\cV) = \mathcal{R} ( \mathrm{d} ^\ast ) \oplus \mathcal{N} ( \mathrm{d} ) , \qquad \ell^2_\la(\cE) = \mathcal{R} ( \mathrm{d} ) \oplus \mathcal{N} ( \mathrm{d} ^\ast ),
\ee
where $\mathcal{R}(\cdot)$ and $\cN(\cdot)$ represent the range and nullspace, and $\oplus$ represents the orthogonal decomposition with respect to the inner products on $\ell^2(\cV)$ and $\ell^2_\la(\cE)$. 

Now given $f$, the equation $\mathrm{d} u = f$ can be solved only if $f \in \mathcal{R} (\mathrm{d})$. In general, a least squares solution to $\mathrm{d}u=f$ instead solves $\mathrm{d}u=f_1$ where $f=f_1+f_2$ with $f_1 \in \mathcal{R} (\mathrm{d})$ and $f_2 \in \mathcal{N} (\mathrm{d}^\ast)$. Applying $\mathrm{d}^\ast$ yields $\mathrm{d}^\ast \mathrm{d} u = \mathrm{d}^\ast f_1 = \mathrm{d}^\ast f$, which is \eqref{Poisson3}. 

We also note that any two solutions $u,v$ to \eqref{Poisson3} differ by a constant if $\cG$ is connected. This is because $\mathrm{L}u = \mathrm{L}v$ implies $\mathrm{d}u = \mathrm{d}v$ ($\mathrm{L}u = 0 \Rightarrow  \langle \mathrm{d}^*\mathrm{d}u, u \rangle = 0 \Rightarrow \langle \mathrm{d}u, \mathrm{d}u \rangle_\la = || \mathrm{d}u ||_{\ell^2_\la(\cE)}^2 = 0 \Rightarrow \mathrm{d}u = 0$). In particular, if $u$ is specified at a state, say $u(O) = u_0$ for some $O \in \cV$, then $u$ is uniquely determined by the equation \eqref{Poisson3}.

Let $(X_t)_{t \in \N_0}$ denote the cooperative process valued in $\cV$. In general, cooperative evolution can be described as any stochastic process satisfying the condition $(X_t, X_{t+1}) \in \overline{\cE}$. However, to understand the connection between the average path integral \eqref{value2} and the Poisson's equation, we now focus on the processes whose governing law $\cP$ belongs to an important class in probability theory, known as the class of reversible Markov chains. 

A Markov chain $(X_t)_{t \in \N_0}$ is characterized by its transition probabilities (or rates) between two adjacent states $S$ and $T$, denoted by $p_{S,T}$, which is the probability of $X_{t+1} = T$ given $X_t = S$. A Markov chain on a graph $\cG$ is called  reversible if there exists an edge weight $\la : \overline{\cE} \to \R_+$ such that the transition probabilities are defined by
 \begin{align}\label{MC2}
p_{S,T} = \frac{\la (S,T)} { \sum_{U \sim S} \la (S,U)}.
\end{align} 
\eqref{MC2} describes a (biased) random walk on the graph, where the parameter $\lambda$ indicates the relative likelihood of the direction in which cooperation is likely to progress, thereby providing increased flexibility in modeling stochastic cooperative processes. This formulation includes the previous coalition formation process \eqref{MC} as a particular case, where $\lambda$ remains constant at 1 and $\mathcal{G}$ represents the coalition graph \eqref{oldG}.

If a Markov chain is reversible, there exists a stationary distribution $\pi=(\pi_S)_{S \in \cV}$ such that $\pi_S p_{S,T} = \pi_T p_{T,S}$ for all $ S, T \in \cV$. Another implication of reversibility is that every loop and its inverse loop have the same probability of being realized, that is,
\be\label{reversibility}
p_{S,S_1} p_{S_1,S_2} \dots  p_{S_{n-1},S_n} p_{S_n, S} =p_{S,S_n} p_{S_n,S_{n-1}} \dots  p_{S_2,S_1} p_{S_1, S}.
\ee
We illustrate the importance of these properties in establishing the following result.
\begin{theorem}\label{main2}
Let the Markov chain \eqref{MC2} be defined on a  connected graph $G$ with weight $\la$, and let $f_i \in \ell^2_\la(\cE)$. Then $\Phi_{f_i}$ in \eqref{value2}  uniquely solves the Poisson's equation
\be\label{ls3}
\mathrm{L}\Phi_{f_i} = \mathrm{d}^* f_i \ \ \text{with the initial condition } \ \Phi_{f_i} (O) = 0.
\ee
\end{theorem}
\begin{proof}
Fix $T \in \cV$ and let $\{T_1,...,T_n\}$ be the set of all states adjacent to $T$ (which is nonempty by the connectedness of $\cG$), and set $\La_T= \sum_{k=1}^n \la (T, T_k) > 0 $. 

By \eqref{div} and \eqref{MC2}, for any $f_i \in \ell_\la^2(\cE)$, we have
   \begin{align}
 \label{divf}
\mathrm{d}^*{f_i} (T) / \La_T &= \sum_{k=1}^n p_{T, T_k}  {f_i}(T_k,T), \text{ and} \\
\label{divV}
\mathrm{L} \Phi_{f_i} (T) / \La_T &=  \sum_{k=1}^n p_{T, T_k}\big(\Phi_{f_i}(T) - \Phi_{f_i}(T_k)\big) =  \sum_{k=1}^n p_{T, T_k} \Phi_{f_i}^{T_k} ( T) 
\end{align}
where the last equality is from \ref{transition} which will be shown later. Now we can interpret the right side of \eqref{divV} as the aggregation \eqref{value2} of path integrals of $f_i$ \eqref{pathintegral2} along all loops beginning and ending at $T$, but in this aggregation of $f_i$ we do not take into account the first move from $T$ to $T_k$, since this first move is described by the transition rate $p_{T, T_k}$ and not driven by $\Phi^{T_k}_{f_i}$. On the other hand, if we aggregate path integrals of $f_i$ for all loops emanating from $T$, we get $0$ due to the reversibility \eqref{reversibility} and the sign changing property of $f_i$. This observation allows us to conclude:
\begin{align*}
0 &= \text{aggregation of path integrals of $f_i$ along all loops emanating from $T$} \\
&= \text{aggregation of path integrals of $f_i$ along all loops except the first moves} \\
&\q + \text{aggregation of path integrals of $f_i$ for all first moves from $T$}\\
&= \sum_{k=1}^n p_{T, T_k} \Phi_{f_i}^{T_k} ( T) +  \sum_{k=1}^n p_{T, T_k} f_i(T,T_k) \\
&= \mathrm{L} \Phi_{f_i} (T) / \La_T - \mathrm{d}^*f_i (T) / \La_T,
\end{align*}
yielding $\mathrm{L} \Phi_{f_i} (T) = \mathrm{d}^*f_i (T)$.
\end{proof}
The theorem allows us to evaluate the potentially intractable value function $\Phi_{f_i}$ by a feasible problem of solving a system of least-squares linear equations \eqref{ls3}. 
\begin{example} We calculate the allocation \eqref{value2} for the $\al$-Shapley value \eqref{modifiedmarginalvalue} with the glove game $v$ \eqref{ex:introGlove}. By Theorem \ref{main2}, we can solve the Poisson's equation
\be
\mathrm{L} \Phi_{f_{\al, i}} = \mathrm{d} ^\ast f_{\al, i} \ \text{ with initial condition }\ \Phi_{f_{\al, i}} (\emptyset) = 0, \q i=1,2,3. 
\ee
  Let us denote the vertices of the unit cube by $n_0 = (0,0,0)$, $n_1 = (1,0,0)$, $n_2 = (0,1,0)$, $n_3 = (0,0,1)$, $n_4 = (1,1,0)$, $n_5 = (1,0,1)$, $n_6 = (0,1,1)$, $n_7 = (1,1,1)$. Then the matrix representation of the gradient $\mathrm{d}$ and the marginal values $f_{\al, 1},f_{\al, 2},f_{\al, 3}$ are given by
\[
 \kbordermatrix{ 
& n_0 & n_1 & n_2 & n_3 & n_4 & n_5 & n_6 & n_7 \\
(n_0, n_1) & -1 & 1 & 0 & 0 & 0 & 0 &0 &0\\
(n_0, n_2) & -1 & 0 & 1 & 0 & 0 & 0 & 0 &0\\
(n_0, n_3) & -1 & 0 & 0 & 1 & 0 & 0 & 0 &0\\
(n_1, n_4) & 0 & -1 & 0 & 0 & 1 & 0 & 0 &0\\
(n_2, n_4) & 0 & 0 & -1 & 0 & 1 & 0 & 0 &0\\
(n_1, n_5) & 0 & -1 & 0 & 0 & 0 & 1 & 0 &0\\
(n_3, n_5) & 0 & 0 & 0 & -1 & 0 & 1 & 0 &0\\
(n_2, n_6) & 0 & 0 & -1 & 0 & 0 & 0 & 1 &0\\
(n_3, n_6) & 0 & 0 & 0 & -1 & 0 & 0 & 1 &0\\
(n_4, n_7) & 0 & 0 & 0 & 0 & -1 & 0 & 0 &1\\
(n_5, n_7) & 0 & 0 & 0 & 0 & 0 & -1 & 0 &1\\
(n_6, n_7) & 0 & 0 & 0 & 0 & 0 & 0 & -1 &1
}
, \q 
 \kbordermatrix{ 
& f_{\al, 1} & f_{\al, 2} & f_{\al, 3} \\
(n_0, n_1) & 0 & 0 & 0\\
(n_0, n_2) & 0 & 0 & 0\\
(n_0, n_3) & 0 & 0 & 0\\
(n_1, n_4) & \frac{1-\al}{2} & \al & \frac{1-\al}{2}\\
(n_2, n_4) & \al & \frac{1-\al}{2} & \frac{1-\al}{2}\\
(n_1, n_5) & \frac{1-\al}{2} & \frac{1-\al}{2} & \al\\
(n_3, n_5) & \al & \frac{1-\al}{2} & \frac{1-\al}{2}\\
(n_2, n_6) & 0 & 0 & 0\\
(n_3, n_6) & 0 & 0 & 0\\
(n_4, n_7) & 0 & 0 & 0\\
(n_5, n_7) & 0 & 0 & 0\\
(n_6, n_7) & \al & \frac{1-\al}{2} & \frac{1-\al}{2}
}.
\]

\noin Then $\mathrm{d} ^\ast$ is represented by the transpose matrix of $\mathrm{d}$ since the edge weight $\la$ is constant $1$. In view of the initial condition, we need to solve $\mathrm{L}_0 u_i = \mathrm{d}^\ast f_{\al, i}$, where $\mathrm{L}_0$ is a $8 \times 7$ matrix equal to $\mathrm{L}$ with the first column removed; then $u_i \in \R^7$ coincides with $\Phi_{f_{\al, i}}$ for each nonempty $S \subset \{1,2,3\}$. Since $u_i$ is unique, it is represented by
\be\label{wformula}
u_i = (\mathrm{L}_0^\ast \mathrm{L}_0)^{-1} \mathrm{L}_0^\ast \mathrm{d}^\ast f_{\al, i}, \q i=1,2,3.
\ee
The extended $\al$-Shapley value allocation table \eqref{alphashapley} is obtained by solving \eqref{wformula}.
\end{example}

\begin{remark}[Stern and Tettenhorst's component games \cite{stern2019hodge}] Given a coalition game $v \in \cG_N$, Stern and Tettenhorst analyzed the Poisson's equation \eqref{Poisson3} on the coalition graph \eqref{oldG} with $\p_i v$ representing player $i$'s marginal value. Specifically, they defined the component games $(v_i)_{i \in N}$ as the solution to the equation
\be\label{StTePoisson}
\mathrm{L} v_i = \mathrm{d}^* \p_i v \ \text{ with } v_i(\es) = 0, \ \text{ for each } i \in N.
\ee
Stern and Tettenhorst utilized this specific Poisson's equation as the basis of their theory. In this study, we presented a microfoundation of the component games $(v_i)_{i \in N}$ by showing $v_i (S) = \Phi_i(v, S)$ for all $i \in N$ and $S  \subset N$, where $\Phi_i$ was defined by \eqref{value}. This was demonstrated in Theorem \ref{main2} in a much broader context, i.e., for any network $\cG$ and players' marginal values $(f_i)_{i \in N}$. Additionally, \citet{stern2019hodge}'s main result establishes $v_i(N) = \phi_i(v)$, that is, each player's component game value at the grand coalition coincides with the Shapley value. This result can now be interpreted as a particular case of Theorems \ref{coincide1} and \ref{main2}. In particular, these theorems enable us to utilize the equation \eqref{StTePoisson} for the proof of Theorem \ref{main}, as presented in Section \ref{proofs}.
\end{remark}

\begin{remark}[From infinite to finite paths]\label{infinitetofinite}
The definition of the value  $\Phi_{f_i}^S (T)$ in \eqref{value2} involves the sum of an infinite number of path integrals from $S$ to $T$, even when the underlying graph has a finite number of nodes and edges. This is because we allow paths to contain loops. However, if two essential conditions are met -- the sign changing property of edge flows $\ell^2_\la(\cE)$ and reversibility \eqref{reversibility} -- we can effectively reduce the sum to a finite number of appropriately weighted paths with no loops. 

We say that a path $(X_0, X_1, \dots, X_\tau)$ has no (internal) loop if $X_0,...,X_\tau$ are all distinct, with the exception of the possibility $X_0 = X_\tau$, i.e., the path itself can be a loop. If a graph $G = (\cV, \cE)$ is finite, then there are finitely many paths with no loops. Let $\cN_{S,T}$ represent the collection of such no-loop paths from $S$ to $T$ (where $S=T$ is possible). For instance, in Figure \ref{figure4}, $\cN_{S,T}$ consists of the following three paths: $(S, T)$, $(S, V, U, T)$, $(S, V, W, U, T)$. Let $\cC_{S,T}$ be the collection of all paths from $S$ to $T$, i.e., $X_0 = S$, $X_\tau = T$, and $X_t \neq T$ for all $ t = 1,2,..., \tau -1$. We now assert that the following equality holds:
\be\label{reduction}
\Phi_{f_i}^S (T) = \sum_{\tilde\omega = (X_0,...,X_\tau) \in \cN_{S,T}} \mu(\tilde \omega) \sum_{t=1}^{\tau} f_i (X_{t-1}, X_t)
\ee
for some probability  $\mu$ on $\cN_{S,T}$, i.e., $\sum_{\tilde\omega \in \cN_{S,T}} \mu(\tilde \omega) = 1$ with $\mu(\tilde \omega) > 0$ for every $\tilde\omega$.

What probability $\mu$ will make \eqref{reduction} hold? Let us say $\omega \in \cC_{S,T}$ and $\tilde \omega \in \cN_{S,T}$ are equivalent, denoted by $\omega \sim \tilde \omega$, meaning that if all internal loops in $\omega$ are removed, it equals $\tilde \omega$.  For example, in Figure \ref{figure4}, the path $(S,V, W, U, V, W, U, T)$ is equivalent to $(S, V, W, U, T)$. Observe that $\sim$ yields a partition of $\cC_{S,T}$, with each $\tilde \omega \in \cN_{S,T}$ representing a partition. We define the probability distribution $\mu$ on $\cN_{S,T}$ as follows:
\be\label{weightassignment}
\mu(\tilde \omega) = \cP(\{ \omega \ | \ \omega \sim \tilde \omega\}) = \int_{\{ \omega\, |\, \omega \sim \tilde \omega\}} d \cP(\omega)
\ee
where $\cP$ is the law of the Markov chain, so for a path $\omega = (X_0,X_1,...,X_\tau)$, we have $d \cP(\omega) = \prod_{t=1}^\tau p_{X_{t-1},X_t} = p_{X_0,X_1}p_{X_1,X_2} \dots p_{X_{\tau-1} X_\tau}$. 

For example, consider the general two person game in Example \ref{twoperson} with initial and terminal states $S=\emptyset$ and $T=\{1\}$. There are two no-loop paths: $\tilde \omega_1 = (\emptyset, \{1\})$ and $\tilde \omega_2 = (\emptyset, \{2\}, \{1,2\}, \{1\})$. If we assign $1/2$ to $\tilde \omega_1$ and $(1/2)^3 = 1/8$ to $\tilde \omega_2$ according to their lengths, it fails to sum to $1$, and it will not yield a proper weight assignment. 

Let $\{X_t\}_{t \ge 0}$ denote the random walk \eqref{MC} with $X_0 = \emptyset$. With probability $1/2$, $X_1 = \{1\}$ or $ \{2\}$. In the latter case $X_1 = \{2\}$, due to the graph symmetry, $X_t$ will eventually arrive at $\{1\}$ via $\emptyset$ or $\{1,2\}$ with equal probability $1/2$. This and \eqref{weightassignment} show $\mu(\tilde \omega_1) = 1/2 + 1/2\cdot 1/2 = 3/4$ and $\mu(\tilde \omega_2) = 1/2\cdot 1/2 = 1/4$. Now the path integral of $\p_1 v $ along $\tilde \omega_1$ is $v_1$, while along $\tilde \omega_2$ is $v_{12}-v_2$, resulting in player $1$'s value at $\{1\}$ as $3/4 \cdot v_1 + 1/4 \cdot (v_{12}-v_2) = (3v_1 + v_{12} - v_2)/4$. Similarly, the path integral of $\p_2 v$ along $\tilde \omega_1$ is $0$, while along $\tilde \omega_2$ is $v_2 + (v_1 - v_{12})$, yielding player $2$'s value at $\{1\}$ as $(v_2 + v_1 - v_{12})/4$. This coincides with \eqref{twopersontable}, verifying \eqref{reduction} in this case. For establishing \eqref{reduction} in general, the sign changing property of edge flows and the reversibility of the Markov chain appear crucial, which implies that the path integral of an edge flow along a loop and its reverse loop must have the opposite sign while having the same probability of the loops being realized, and thus cancel out in the sum \eqref{reduction}. If, on the other hand, at least one of the two conditions fails, the author does not expect the reduction \eqref{reduction} to hold in general. This further signifies the importance of the two conditions, and their interplay.
\end{remark}

\section{Conclusion}\label{Conclusion}

This study aims to present a novel mathematical framework for cooperative games that goes beyond existing setups in the literature. Previous research on cooperative value allocation has primarily considered it as determined by a number of axioms. Theorem \ref{main} can be interpreted as a continuation of this direction. However, the primary objective of this study is to propose a shift from the current axiomatic framework to a stochastic path-integral framework. While the scope of cooperative games whose value can be determined by the axiomatic framework may be limited, often necessitating a restrictive set of assumptions for valuation, and the resulting value is implicit, as evidenced in this paper, the modeling of players' cooperative progress as a random process in cooperative graphs and the definition of value as the average path integral of players' marginal values allows for the seamless incorporation of a broad range of cooperative graphs, progressions, games, and marginal values. Moreover, the value is explicitly represented in this framework. We believe that this change in perspective will enable the application of cooperative game theory to various scientific domains, particularly in machine learning and AI, where the classical Shapley value already holds significance. In the realm of economic applications, as alluded in Remark \ref{bettersetup}, investigating the specification of the game network $\mathcal{G}$ and cooperative process law $\mathcal{P}$, as well as establishing players' marginal values $\bf f$ in a manner tailored to the specific economic context at hand, presents an intriguing avenue for future research.

\section{Proof of Theorem \ref{main} and Equation \ref{transition}}\label{proofs}

\begin{proof}[Proof of Theorem \ref{main}] 
First, we claim that A1--A5 determines the operator $\Phi$ uniquely (if exists). For each player set $N$, define games $\delta_{S,N} \in \cG_N$ for each $ \es \neq S \subset N$ by
\[
\delta_{S,N}(S) = 1, \q 
      \delta_{S,N}(T)= 0 \,\text{ if } \, T \neq S.
\]
We proceed by an induction on $|N|$. The case $|N|=1$ is already from A1. Suppose the claim holds for $|N|-1$, so $\Phi_j(\delta_{S, N \setminus \{i\}}, \cdot)$ are determined for all $i,j \in N$ and $\es \neq S \subset N \setminus \{i\}$. Define the games $\Delta_{(S, S \cup \{i\})} \in \cG_N$ for each $\es \neq S \subset N \setminus \{i\}$ by 
\[
\Delta_{(S, S \cup \{i\})}(T) = 1 \, \text{ if } \, T=S \text{ or } T=S \cup \{i\}, \q \Delta_{(S, S \cup \{i\})}(T) =0 \, \text{ otherwise.}
\]
Notice then A4 (and induction hypothesis) determines $\Phi$ for all $\Delta_{(S, S \cup \{i\})} \in \cG_N$. Then thanks to A2, to prove the claim, it is enough to show that A1--A5 can determine $\Phi$ for the pure bargaining game $\delta:= \delta_{N,N}$, because for any $\es \neq S \subset N$, we can write $\delta_{S,N}$ as the following sign-alternating sum
\be
\delta_{S,N} = \Delta_{(S, S \cup \{i_1\})} - \Delta_{(S \cup \{i_1\}, S \cup \{i_1, i_2\})} + \Delta_{(S \cup \{i_1, i_2\}, S \cup \{i_1, i_2, i_3\})} - \dots \pm \delta_{N,N}. \nn
\ee
By A3, $\sum_{S \subset N} \Phi_i(\delta,S)$ is constant for all $i \in N$, thus equals $ 1/|N|$ by A1. Define 
\[
u_i (S) := \Phi_i(\delta,S) - \frac{1}{|N| 2^{|N|}} \ \text{ for all } \ S \subset N
\]
so that $u_i (\emptyset) = -\frac{1}{|N| 2^{|N|}}$ and $\sum_{S \subset N} u_i(S)=0$ for all $i$. Now observe A5 implies:
\be
u_i(S) + u_i(S \cup \{i\}) \text{ is constant for all } S \subset N \setminus \{i\}, \text{ hence it is zero.} \nn
\ee
This determines $u_i$, thus $\Phi_i(\delta, \cdot)$, as follows: suppose $u_i(S)$ has been determined for all $i$ and $|S| \le k-1$. Let $|T| =k \le |N|-1$. Then we have $u_i(T) = - u_i (T \setminus \{i\})$ for all $i \in T$ and it is constant (say $c_k$) by A3. Using A1 and A3, we obtain
\be
0 = \delta(T)=\sum_{i\in N}\Phi_i(\delta, T)=\sum_{i\in N}\bigg(u_i(T) + \frac{1}{|N| 2^{|N|}} \bigg), \nn
\ee
yielding $\sum_{i\in N} u_i(T) = -1/ 2^{|N|}$. With $u_i(T) = c_k$ for all $i \in T$, we deduce $u_j(T) =  \frac{ -1 - 2^{|N|} k c_k } {2^{|N|}(|N|-k) }$ for all $j \notin T$. This shows $u_i(T)$ is determined for all $|T|=k \le |N|-1$. Of course, $\Phi_i(\delta,N) = 1/|N|$ for all $i \in N$ by A1 and A3. By induction (on $|N|$ and on $k$ for each $N$), the proof of uniqueness of the operator $\Phi$ is therefore complete.

It remains to show that $\Phi$ represented by the average path integral \eqref{value} satisfies A1--A5. As remarked earlier, Theorems \ref{coincide1} and \ref{main2} show that $v_i (S) = \Phi_i(v, S)$ for all $S \subset N$, where $(v_i)_{i \in N}$ is the unique solution to the Poisson's equation
\be\label{ls1}
\mathrm{L} v_i = \mathrm{d}^* \p_i v \ \text{ with } v_i(\es) = 0, \ \text{ for each } i \in N.
\ee
We will henceforth show that $( v_i )_{ i \in N }$ in place of $(\Phi_i(v, \cdot))_{i \in N}$ satisfies A1--A5. Firstly, A2 is clearly satisfied. To show that A1 is satisfied, we compute 
  \begin{equation*}
    \mathrm{L} \sum _{ i \in N } v _i = \sum _{ i \in N}  \mathrm{L} v _i = \sum _{ i \in N }   \mathrm{d}^\ast \p_i v =   \mathrm{d}^\ast \sum _{ i \in N } \p_i v =   \mathrm{d}^\ast \mathrm{d} v = \mathrm{L}v,
  \end{equation*}
since $ \mathrm{d} = \sum _{ i \in N } \p_i $. Hence by unique solvability of \eqref{ls1}, $\sum _{ i \in N } v _i =v$ as desired. 

Next, let $\sigma$ be a permutation of $N$. Let $\sigma$ act on $\ell ^2 (2^{N})$ and $\ell ^2 (\cE)$ via
\be
\sigma v(S) = v(\sigma (S)) \ \text{and} \ \sigma f\bigl( S , S \cup \{ i \}\bigr) =  f\bigl( \sigma (S) , \sigma (S \cup \{ i \})\bigr),  \ v \in \ell ^2 (2^{N}), \ f \in \ell ^2 (\cE).  \nn
\ee
It is easy to check $ \mathrm{d} \sigma = \sigma \mathrm{d} $ and
  $ \mathrm{d} _i \sigma  = \sigma \mathrm{d} _{ \sigma (i)
  } $. We also have $\mathrm{d}^\ast \sigma = \sigma \mathrm{d}^\ast$, since 
\be
  \langle v , \mathrm{d}^\ast \sigma f \rangle =    \langle \mathrm{d} v ,  \sigma f \rangle =  \langle \sigma^{-1} \mathrm{d} v ,   f \rangle =    \langle  \mathrm{d}  \sigma^{-1}  v ,   f \rangle =  \langle   \sigma^{-1}  v ,  \mathrm{d}^\ast f \rangle=   \langle   v ,  \sigma  \mathrm{d}^\ast f \rangle \nn
\ee
for any $v \in \ell ^2 (2^{N})$, $f \in \ell ^2 (\cE)$. Now let $\sigma$ be the transposition of $i,j$. We have
\be
   \mathrm{L}(\sigma v)_i = \mathrm{d}^\ast \p_i \sigma v = \mathrm{d}^\ast  \sigma  \p_j v =   \sigma \mathrm{d}^\ast \p_j v = \sigma  \mathrm{L} v_j =  \mathrm{L}\sigma v_j \nn
\ee
which shows $(\sigma v)_i = \sigma v_j$ by the unique solvability. Notice this corresponds to A3.
  
For A4, let $v \in \cG_N$, $i \in N$, and assume $\p_i v = 0$. Then from \eqref{ls1} we readily get $v_i \equiv 0$. Fix $j \neq i$, and let $\mathrm{\tilde d}$, ${\tilde \p}_j$ be the differential operators restricted on $2^{N \setminus \{i\}}$, and set $\tilde v = v_{-i}$, i.e., $\tilde v$ is the restriction of $v$ on $2^{N \setminus \{i\}}$. Let $\tilde v_j : 2^{N \setminus \{i\}} \to \R$ be the solution to the equation $\mathrm{\tilde d}^\ast \mathrm{\tilde d} \tilde v_j = \mathrm{\tilde d}^\ast {\tilde \p}_j \tilde v$ with $\tilde v_j (\es) = 0$. Finally, in view of A4, define $v_j \in \cG_N$ by $v_j = \tilde v_j$ on $2^{N \setminus \{i\} }$ and $\p_i v_j = 0$. Now observe that A4 will follow if we can verify that this $v_j$ indeed solves the equation $\mathrm{ d}^\ast \mathrm{d}  v_j = \mathrm{ d}^\ast \p_j  v$.

To show this, let $S \subset N \setminus \{i\}$. In fact the following string of equalities holds:
\be
\mathrm{d}^\ast \mathrm{d} v_j (S \cup \{i\}) 
=\mathrm{d}^\ast \mathrm{d} v_j (S) 
=\mathrm{\tilde d}^\ast \mathrm{\tilde d} \tilde v_j (S) 
=\mathrm{\tilde d}^\ast {\tilde \p}_j \tilde v (S) 
=\mathrm{d}^\ast \p_j  v (S) 
=\mathrm{d}^\ast \p_j  v (S \cup \{i\})
\nn
\ee
which simply follows from the definition of the differential operators. For instance
\begin{align*}
\mathrm{d}^\ast \mathrm{d} v_j (S)   
= \sum _{ T \sim S } \mathrm{d} v_j ( T, S ) 
= \sum _{ T \sim S, \, T \neq S \cup \{i\} } \mathrm{d} v_j ( T, S )
=\mathrm{\tilde d}^\ast \mathrm{\tilde d} \tilde v_j (S)
\end{align*}
where the second equality is due to $\p_i v_j = 0$. On the other hand, since $j \neq i$,
\begin{align*}
\mathrm{d}^\ast \p_j v (S)   
= \sum _{ T \sim S } \p_j v ( T, S ) 
=  \sum _{ T \sim S } {\tilde \p}_j \tilde v ( T, S ) 
=\mathrm{\tilde d}^\ast {\tilde \p}_j \tilde v (S).
\end{align*}
The first equality is due to the definition of $v_j$ (i.e. $v_j = \tilde v_j$ on $2^{N \setminus \{i\} }$ and $\p_i v_j = 0$), and the last equality is due to $\p_i v= 0$. This verifies A4.

Finally, we verify A5. For this, we need to verify the following claim:
\be\label{constancycondition}
v_i (S) + v_i (S \cup \{i\}) \, \text{ is constant over all } \, S \subset N \setminus \{i\}. 
\ee
Let $S \subset N \setminus \{i\}$, and recall $\mathrm{d}^\ast \p_i v (S) = v(S) - v(S \cup \{i\}) = -\mathrm{d}^\ast \p_i v (S \cup \{i\})$. Hence, $\mathrm{L} v_i (S) +\mathrm{L} v_i (S \cup \{i\} ) = 0$. Define $w_i \in \ell^2(2^{N})$ by $w_i (S) = v_i(S \cup \{i\} )$ and $w_i (S \cup \{i\} )= v_i(S  )$ for all $S \subset N \setminus \{i\}$. Then clearly $ \mathrm{L} v_i (S \cup \{i\} ) = \mathrm{L} w_i (S )$ and $\mathrm{L} v_i (S) = \mathrm{L} w_i (S \cup \{i\}  )$. Thus $\mathrm{L} (v_i + w_i) \equiv 0$, hence $v_i + w_i \in  \mathcal{N} ( \mathrm{d} )$, meaning that $v_i + w_i$ is constant. This proves the claim, hence the theorem.
\end{proof}

The following proof of the transition formula \eqref{transition} is required for the proof of Theorem \ref{main2}. The fact Markov chain visits each state infinitely many times is implicitly used. 

 \begin{proof}[Proof of \eqref{transition}] We firstly show   $\Phi_f^S(S) = 0$ which appears as the initial condition in \eqref{ls3}. Then we show $\Phi_f^S(T) = - \Phi_f^T(S)$, and finally $ \Phi^U_f(T) - \Phi^U_f(S) = \Phi_f^S(T)$.
 
To see $\Phi_f^S(S) = 0$, consider a general sample path $\omega$ starting and ending at $S$, without visiting $S$ along the way. In other words, $\omega$ is a loop emanating from $S$. Let $\omega^{-1}$ denote the reversed path of $\omega$, that is, if $\omega$ visits $T_0 \to T_1 \to \dots \to T_\tau$ (where $T_0 = T_\tau = S$ if $\omega$ is a loop), then $\omega^{-1}$ visits $T_\tau \to \dots \to T_0$. Observe
\begin{align*}
0 &= {\cal I}_f^S(S)(\omega) - {\cal I}_f^S(S)(\omega) \\
&= \sum_{t=1}^{\tau_{S}(\omega)} f  \big(X^S_{t-1}(\omega), X^S_t(\omega) \big) - \sum_{t=1}^{\tau_{S}(\omega)} f  \big(X^S_{t-1}(\omega), X^S_t(\omega) \big) \\
&= \sum_{t=1}^{\tau_{S}(\omega)} f  \big(X^S_{t-1}(\omega), X^S_t(\omega) \big) + \sum_{t=1}^{\tau_{S}(\omega)} f  \big(X^S_{t}(\omega), X^S_{t-1}(\omega) \big) \\
&= \sum_{t=1}^{\tau_{S}(\omega)} f  \big(X^S_{t-1}(\omega), X^S_t(\omega) \big) + \sum_{t=1}^{\tau_{S}(\omega^{-1})} f  \big(X^S_{t-1}(\omega^{-1}), X^S_{t}(\omega^{-1}) \big).
\end{align*}
Let $\cP(\omega) = p_{T_0, T_1}p_{T_1, T_2}\dots p_{T_{\tau-1}, T_\tau}$ denote the probability of the sample path $\omega$ being realized. Reversibility \eqref{reversibility} implies $\cP(\omega) = \cP(\omega^{-1})$ for any loop $\omega$. And there is an obvious one-to-one correspondence between a loop $\omega$ and its reverse $\omega^{-1}$. This implies
\begin{align*}
0 
&= \int_\omega \sum_{t=1}^{\tau_{S}(\omega)} f  \big(X^S_{t-1}(\omega), X^S_t(\omega) \big) d \cP(\omega) + \int_\omega \sum_{t=1}^{\tau_{S}(\omega^{-1})} f  \big(X^S_{t-1}(\omega^{-1}), X^S_t(\omega^{-1}) \big)  d \cP(\omega) \\
&= \int_\omega \sum_{t=1}^{\tau_{S}(\omega)} f  \big(X^S_{t-1}(\omega), X^S_t(\omega) \big) d \cP(\omega) + \int_{\omega} \sum_{t=1}^{\tau_{S}(\omega^{-1})} f  \big(X^S_{t-1}(\omega^{-1}), X^S_t(\omega^{-1}) \big)  d \cP(\omega^{-1}) \\
&= \Phi_f^S(S) + \Phi_f^S(S)
\end{align*}
yielding $\Phi_f^S(S) = 0$ as desired.

Next, we will show $\Phi_f^S(T) = - \Phi_f^T(S)$ for $S \ne T$. Consider a general finite sample path $\omega$ of the Markov chain \eqref{MC} starting at $S$, visiting $T$, then returning to $S$ (this happens with probability 1). We can split this journey into four subpaths:\\
$\omega_1$: the path returns to $S$ $m \in \N \cup \{0\}$ times without visiting $T$,\\
$\omega_2$: the path begins at $S$ and ends at $T$ without returning to $S$,\\
$\omega_3$: the path returns to $T$ $n \in \N  \cup \{0\}$ times without visiting $S$,\\
$\omega_4$: the path begins at $T$ and ends at $S$ without returning to $T$.\\
\begin{figure}
\centering
\begin{subfigure}{0.8\textwidth}
  \centering
  \includegraphics[width=0.9\linewidth]{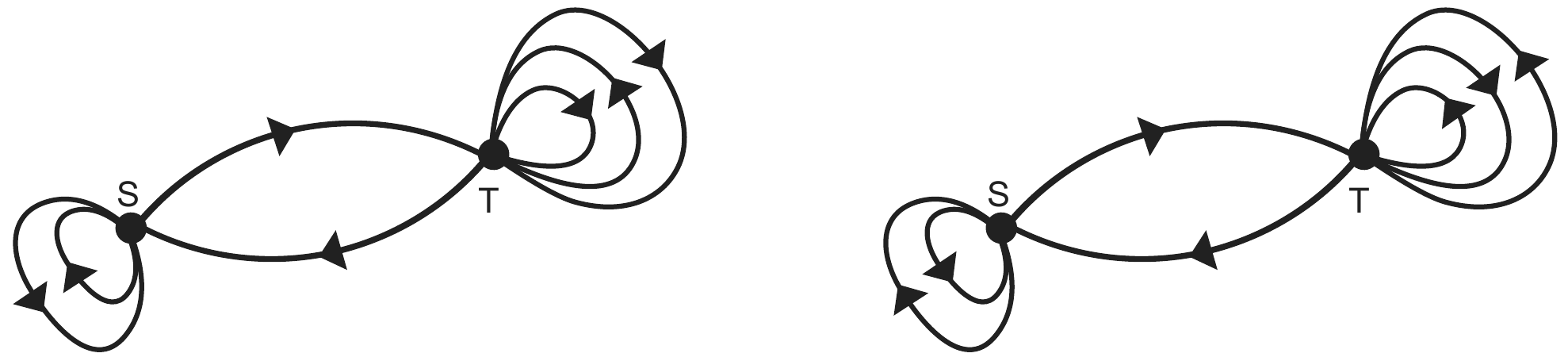}
\end{subfigure}%

\caption{$\omega = \omega_1 \circ \omega_2  \circ \omega_3 \circ \omega_4$ and its pair $\omega' = \omega^{-1}_1  \circ \omega_2  \circ \omega_3^{-1}  \circ \omega_4$.\\ By reversibility, they have the same probability of being realized.
}
\label{figure6}
\end{figure}
Thus $\omega = \omega_1 \circ \omega_2  \circ \omega_3 \circ \omega_4$ is the concatenation of the $\omega_i$'s, and the probability $\cP(\omega)$ of this finite sample path being realized satisfies $\cP(\omega) = \cP(\omega_1)\cP(\omega_2)\cP(\omega_3)\cP(\omega_4)$.

Define a pairing $\omega'$ of $\omega$ by $\omega' := \omega^{-1}_1  \circ \omega_2  \circ \omega_3^{-1}  \circ \omega_4$. This is another general sample path starting at $S$, visiting $T$, then returning to $S$. Then we have $\cP(\omega) = \cP(\omega')$ because $\cP(\omega_1) = \cP(\omega^{-1}_1)$ and $\cP(\omega_3) = \cP(\omega^{-1}_3)$ by \eqref{reversibility}, and moreover,
\be
{\cal I}_f^S(T)(\omega) + {\cal I}_f^S(T)(\omega') = 2 \sum_{t=1}^{\tau_{T}(\omega_2)} f  \big(X^S_{t-1}(\omega_2), X^S_t(\omega_2) \big), \nn
\ee
because the loops $\omega_1$ and $\omega^{-1}_1$ aggregate $f$ with opposite signs, hence they cancel out in the above sum. Now consider $\tilde \omega := \omega_3 \circ \omega_2^{-1} \circ \omega_1 \circ \omega_4^{-1}$ and $\tilde \omega' := \omega_3^{-1} \circ \omega_2^{-1} \circ \omega_1^{-1} \circ \omega_4^{-1}$. $(\tilde \omega, \tilde \omega')$ represents a pair of general sample paths starting at $T$, visiting $S$, then returning to $T$. We then deduce
\begin{align*}
{\cal I}_f^T(S)(\tilde \omega) + {\cal I}_f^T(S)(\tilde \omega') &= 2 \sum_{t=1}^{\tau_{S}(\omega_2^{-1})} f  \big(X^T_{t-1}(\omega_2^{-1}), X^T_t(\omega_2^{-1}) \big)  \\
&= -2 \sum_{t=1}^{\tau_{T}(\omega_2)} f \big(X^S_{t-1}(\omega_2), X^S_t(\omega_2) \big)  \\
&= - \big({\cal I}_f^S(T)(\omega) + {\cal I}_f^S(T)(\omega')\big)
\end{align*}
because $f(U,V)= -f(V,U)$ for any edge $(U,V)$. Due to the one-to-one correspondence between the paths $\omega, \omega',\tilde \omega, \tilde \omega'$, and $\cP(\omega) = \cP(\omega') = \cP(\tilde \omega) = \cP(\tilde \omega')$ from  \eqref{reversibility}, the desired identity $\Phi_f^S(T) = - \Phi_f^T(S)$ now follows by integration:
\begin{align*}
\int_\omega \big[ {\cal I}_f^T(S)(\tilde \omega) + {\cal I}_f^T(S)(\tilde \omega') \big] d\cP(\omega)  
&= \int_{ \omega} {\cal I}_f^T(S)(\tilde \omega) d\cP(\tilde \omega) + \int_{ \omega} {\cal I}_f^T(S)(\tilde \omega') d\cP(\tilde \omega')\\
&= 2 \Phi_f^T(S),
\end{align*}
and similarly, 
\begin{align*}
\int_\omega \big[ {\cal I}_f^S(T)(\omega) + {\cal I}_f^S(T)(\omega') \big] d\cP(\omega) =  2 \Phi_f^S(T).
\end{align*}

Finally, to show $ \Phi_f^U(T) - \Phi_f^U (S) = \Phi^S_f(T)$ for distinct $S,T,U$, we proceed 
\begin{align*}
&{\cal I}_f^U(T) - {\cal I}_f^U (S) = \sum_{t=1}^{\tau_{T}} f \big( X^U_{t-1}, X^U_t \big) 
-  \sum_{t=1}^{\tau_{S}} f  \big( X^U_{t-1}, X^U_t \big) \\
&= {\bf 1}_{\tau_{S} < \tau_{T}}\sum_{t=\tau_{S} + 1}^{\tau_{T}} f  \big( X^U_{t-1}, X^U_t \big) 
- {\bf 1}_{\tau_{T} < \tau_{S}}\sum_{t=\tau_{T} + 1}^{\tau_{S}} f \big( X^U_{t-1}, X^U_t \big).
\end{align*}
By taking expectation, we obtain the following via the Markov property
\begin{align*}
\E[{\cal I}_f^U(T)]  - \E[{\cal I}_f^U (S)] &= \cP(\{\tau_{S} < \tau_{T} \}) \Phi^S_f(T) 
- \cP(\{\tau_{T} < \tau_{S}\}) \Phi^T_f(S) \\
&=\big( \cP(\{\tau_{S} < \tau_{T} \}) \Phi^S_f(T) 
+ \cP(\{\tau_{T} < \tau_{S}\})\big) \Phi^S_f(T) = \Phi^S_f(T),
\end{align*}
which proves the transition formula  $ \Phi_f^U(T) - \Phi_f^U (S) = \Phi^S_f(T)$.
\end{proof}

\noin {\bf Author's statement:} This study does not involve any data. The author declares that he does not have any conflict-of-interest concerning this manuscript or its contents.

\bibliographystyle{plainnat}
\bibliography{TLbib}

\end{document}